\documentclass[10pt]{article}
\usepackage[ansinew]{inputenc}
\usepackage{amsfonts,amsthm,amssymb,amsmath}
\usepackage{enumerate}
\usepackage{cite}
\usepackage{calc}
\usepackage{ifthen}
\usepackage{graphicx}
\usepackage{bbm}
\usepackage{multirow}
\usepackage[centerlast]{caption}
\usepackage{mathtools}
\mathtoolsset{showonlyrefs}

\usepackage{geometry}
\usepackage{tikz}
\usetikzlibrary{decorations.pathmorphing,decorations.pathreplacing}
\usepackage{theoremref}
\newtheorem{thm}{Theorem}[subsection]
\newtheorem{lem}[thm]{Lemma}
\newtheorem{cor}[thm]{Corollary}

\newtheorem{rmk}[thm]{Remark}
\newtheorem{ntn}[thm]{Notation}
\newtheorem{dfn}[thm]{Definition}

\allowdisplaybreaks[4]

\makeatletter
  
  \@addtoreset{equation}{section}
\makeatother

\newcommand{\R}{\mathbbm{R}}

\newcommand{\Z}{\mathbbm{Z}}
\newcommand{\filt}{\mathcal}
\newcommand{\indic}{\mathbbm1}
\newcommand{\N}{\mathbbm N}
\newcommand{\Nnull}{{\mathbbm N_0}}
\renewcommand{\P}{\mathbbm{P}}

\newcommand{\calZ}{\mathcal Z}

\newdimen\wcx
\newsavebox\wccbox
\newsavebox\wcabox
\def\widecheck#1{\mathchoice{
\sbox\wccbox{\ensuremath{\displaystyle#1}}%
\sbox\wcabox{\raisebox{1pt}{\smash{\rotatebox[origin=c]{180}{$\displaystyle\widehat{\mbox{\vrule height 0pt width \wd\wccbox}}$}}}}%
\wcx=\ht\wccbox%
\divide\wcx by 4%
\advance\wcx by \wd\wccbox%
\advance\wcx by -\wd\wcabox%
\divide\wcx by 2%
\hbox to 0pt{\hskip\wcx\raisebox{\ht\wccbox}{\usebox\wcabox}\hss}%
\usebox\wccbox}{%
\sbox\wccbox{\ensuremath{\textstyle#1}}%
\sbox\wcabox{\raisebox{0.5pt}{\smash{\rotatebox[origin=c]{180}{$\textstyle\widehat{\mbox{\vrule height 0pt width \wd\wccbox}}$}}}}%
\wcx=\ht\wccbox%
\divide\wcx by 4%
\advance\wcx by \wd\wccbox%
\advance\wcx by -\wd\wcabox%
\divide\wcx by 2%
\hbox to 0pt{\hskip\wcx\raisebox{\ht\wccbox}{\usebox\wcabox}\hss}%
\sbox\wccbox{\ensuremath{\textstyle#1}}%
\usebox\wccbox}{%
\sbox\wccbox{\ensuremath{\scriptstyle#1}}%
\sbox\wcabox{\raisebox{0.3pt}{\smash{\rotatebox[origin=c]{180}{$\scriptstyle\widehat{\mbox{\vrule height 0pt width \wd\wccbox}}$}}}}%
\wcx=\ht\wccbox%
\divide\wcx by 4%
\advance\wcx by \wd\wccbox%
\advance\wcx by -\wd\wcabox%
\divide\wcx by 2%
\hbox to 0pt{\hskip\wcx\raisebox{\ht\wccbox}{\usebox\wcabox}\hss}%
\usebox\wccbox}{%
\sbox\wccbox{\ensuremath{\scriptscriptstyle#1}}%
\sbox\wcabox{\raisebox{0.2pt}{\smash{\rotatebox[origin=c]{180}{$\scriptscriptstyle\widehat{\mbox{\vrule height 0pt width \wd\wccbox}}$}}}}%
\wcx=\ht\wccbox%
\divide\wcx by 4%
\advance\wcx by \wd\wccbox%
\advance\wcx by -\wd\wcabox%
\divide\wcx by 2%
\hbox to 0pt{\hskip\wcx\raisebox{\ht\wccbox}{\usebox\wcabox}\hss}%
\usebox\wccbox}
}

\newlength{\ruck}
\newlength{\vorr}
\DeclareMathOperator*{\bigcupdot}{\mathchoice{%
			\setlength{\ruck}{\widthof{\LARGE$\cup$}*\real{-0.5}+\widthof{\LARGE$^\cdot$}*\real{-0.5}}%
			\setlength{\vorr}{\ruck*\real{-1}+\widthof{\LARGE$^\cdot$}*\real{-1}}%
			\raisebox{-0.5ex}{\hbox{\LARGE{$\cup\hspace{\ruck}^\cdot\hspace{\vorr}$}}}}{%
			\setlength{\ruck}{\widthof{\Large$\cup$}*\real{-0.5}+\widthof{\Large$^\cdot$}*\real{-0.5}}%
			\setlength{\vorr}{\ruck*\real{-1}+\widthof{\Large$^\cdot$}*\real{-1}}%
			\vphantom{\sum}\smash{\raisebox{-0.35ex}{\hbox{\Large{$\cup\hspace{\ruck}^\cdot\hspace{\vorr}$}}}}}{%
			\setlength{\ruck}{\widthof{\normalsize$\cup$}*\real{-0.5}+\widthof{\normalsize$^\cdot$}*\real{-0.5}}%
			\setlength{\vorr}{\ruck*\real{-1}+\widthof{\normalsize$^\cdot$}*\real{-1}}%
			\vphantom{\sum}\smash{\raisebox{-0.2ex}{\hbox{\normalsize{$\cup\hspace{\ruck}^\cdot\hspace{\vorr}$}}}}}{%
			\setlength{\ruck}{\widthof{\footnotesize$\cup$}*\real{-0.5}+\widthof{\footnotesize$^\cdot$}*\real{-0.5}}%
			\setlength{\vorr}{\ruck*\real{-1}+\widthof{\footnotesize$^\cdot$}*\real{-1}}%
			\vphantom{\sum}\smash{\raisebox{-0.2ex}{\hbox{\footnotesize{$\cup\hspace{\ruck}^\cdot\hspace{\vorr}$}}}}}%
}%

\newlength{\pfeilbreite}
\newcommand{\Bijmapp}{\setlength{\pfeilbreite}{\widthof{\normalsize$\nearrow$}}\nearrow\hspace{-\pfeilbreite}\swarrow\!}%
\newcommand{\Bijmap}{\setlength{\pfeilbreite}{\widthof{\normalsize$\nearrow$}*\real{0.8}}\nearrow\hspace{-\pfeilbreite}\swarrow}%

\newcommand{\RL}[2][]{#1\lfloor\frac{#2}2#1\rfloor}
\newcommand{\FF}[2][]{#1\lfloor#2\rfloor}
\newcommand{\SL}[2][]{#2\bmod2}
\DeclareMathOperator{\UR}{\Rsh\!}
\newcommand{\plannedwidth}{\beta}
\newcommand{\ROO}{\lambda}

\DeclareMathOperator{\levelof}{\ell}

\newcommand{\lunten}{\underline \levelof}
\newcommand{\loben}{\overline \levelof}

\newcommand{\closure}{\overline}
\newcommand{\configx}{\bar{\mathbf x}}
\newcommand{\configX}{\mathbf x}
\newcommand{\configy}{\bar{\mathbf y}}

\DeclareMathOperator{\Btwn}{B'twn}
\DeclareMathOperator{\Srnd}{S'rnd}
\DeclareMathOperator{\IG}{InitGrid}

\DeclareMathOperator{\Lane}{Lane}
\DeclareMathOperator{\SG}{SG}

\DeclareMathOperator{\ABC}{ABC}

\begin{document}
%
\author{Hadrian Heil\footnote{supported by the German--Israeli--Foundation, grant number I-974-152.6/2007}}
\title{A Stationary, Mixing and perturbative\\ Counterexample to the $0$-$1$-law\\ for Random Walk in Random Environment\\ in Two Dimensions}
\maketitle
\begin{abstract}
We construct a two-dimensional counterexample of a random walk in random environment (RWRE). The environment is stationary, mixing and $\varepsilon$--perturbative, and the corresponding RWRE has non-trivial probability to wander off to the upper right. This is in contrast to the $0$-$1$-law that holds for i.i.d.\ environments.
\end{abstract}
\section{Random walk in random environment}\label{randwalk}
We start by fixing the notation and the basic notions of the model. 

We work in the $d$-dimensional space $\Z^d$, $d\ge1$. $\Nnull:=\{0,1,2,\dots\}$ and $\N:=\{1,2,\dots\}$ stand for the natural numbers.

We will count dimensions from $0$ to $d-1$; so, we write $u=(u_0,u_1,\dots,u_{d-1})\in\Z^d$, and denote by $e_0,\dots e_{d-1}$ the canonical unit vectors in $\Z^d$. This nonstandard--notation will simplify things later. For two vectors $v,w\in\Z^d$, $v\cdot w$ denotes the scalar product.

For any real number $r\in\R$, we will be using the floor function $\lfloor r\rfloor:=\max\{m\in\Nnull:m\le r\}$ and for any natural number $l\in\Nnull$ the modulo operation $\SL{l}:=\indic_{l\text{ is impair}}\in\{0,1\}$.

If $\P$ is a probability measure, with the convenient notational abuse common in mathematical physics, we write ``$\P$'' for the expectation operator as well. 

Define 
\begin{equation}
 S^d:=\Big\{\varpi\in[0,1]^{\{\pm e_j,0\le j< d\}}:\smashoperator[r]{\sum_{e\in\{\pm e_j,0\le j<d\}}}\quad\varpi(e)=1\Big\},\ d\in\N,
\end{equation}
the set of nearest neighbour transition probabilities on $\Z^d$. We call a family $\omega=(\omega_u)_{u\in\Z^d}$ of $S^d$-valued random variables on an appropriate probability space $(\Omega,\filt A,P)$ a \emph{random environment} on $\Z^d$.

One might ask for a random environment to satisfy, with $0\le\kappa<1/2$ some \emph{ellipticity constant}, the condition
\begin{equation}
 P\big(\omega_u(e)\in(\kappa,1-\kappa)\big)=1\text{ for all }u\in\Z^d,e\in\{\pm e_j,0\le j<d\}.\label{ellipticity}
\end{equation}
If \eqref{ellipticity} is satisfied with $\kappa=0$, the environment is called \emph{elliptic}, and if it is even satisfied with some $\kappa>0$, \emph{uniformly elliptic}.

A morally even stronger notion of homogeneity is reached when one pushes $\kappa$ towards $\frac1{2d}$. For $\varepsilon>0$, $\omega$ is called \emph{$\varepsilon$--perturbative} if 
\begin{equation}
 P\big(\omega_u(e)\in[1/2d-\varepsilon,1/2d+\varepsilon]\big)=1\text{ for all }u\in\Z^d,\ e\in\{\pm e_j,0\le j<d\}.
\end{equation}

We use the term \emph{totally ergodic} for ``ergodic with respect to any shift''.

Take a starting point $v\in\Z^d$. To a random environment $\omega$ on $(\Omega,\filt A,P)$, we associate the random probability measure $P_v^\omega$, which, together with the $\Z^d$-valued random variables $(X_t)_{t\in\Nnull}$, establishes the \emph{random walk in random environment} $(P,P_v^\omega,(X_t)_{t\in\Nnull})$. It is defined to satisfy the Markov-property and 
\begin{gather}
 P_v^\omega(X_0=v)=1, \label{Pomega(X)}\\ 
 P_v^\omega(X_{t+1}=X_t+e|X_t=u)=\omega_u(e),\ e\in\{\pm e_j,0\le j\le d-1\},\ u\in\Z^d.
\end{gather}
In \cite{Kalikow81}, Kalikow considered questions of recurrence and transience of this model, and proved that for uniformly elliptic i.i.d.--environments, 
\begin{equation}
 PP_0^\omega(X_t\cdot v\text{ changes sign infinitely often})\in\{0,1\},\ v\in\Z^d.\label{changesign}
\end{equation}
He also raised the question whether in $d=2$, it holds that
\begin{equation}
 PP_0^\omega(X_t\cdot v\xrightarrow[t\to\infty]{}\infty)\in\{0,1\},\ v\in\R^d\setminus\{0\}.\label{Kalikowsquestion}
\end{equation}
Sznitman and Zerner highlighted in \cite{SznitmanZerner99} that Kalikow's question \eqref{Kalikowsquestion} is valid in any dimension $d\ge2$. They also pointed out that \eqref{changesign} implies
\begin{equation}
 P\big(P_0^\omega(X_t\cdot v\text{ is transient})\big)\in\{0,1\},\ v\in\Z^d.
\end{equation}
The term \emph{Kalikow's $0$--$1$--law} has since been established for this assertion. 

For $d=2$, Zerner and Merkl answer Kalikow's question (positively) for elliptic i.i.d.--envi\-ron\-ments in \cite{ZernerMerkl01}; an improved version of the proof is given in \cite{Zerner07}. Holmes and Salisbury treat the same questions without the assumption of ellipticity in \cite{HolmesSalisbury11}. 

The necessity of the i.i.d.--assumption is assessed in \cite{ZernerMerkl01} by means of an example for $d=2$ of an elliptic, ergodic and stationary environment that features
\begin{equation}
 PP_0^\omega(X_t\cdot v\xrightarrow[t\to\infty]{}\infty)\not\in\{0,1\}\text{ for some }v\in\Z^d.\label{notin}
\end{equation}
\cite{Zerner07} gives a similar example with an even totally ergodic environment.

As for $d\ge3$, Bramson, Zeitouni and Zerner \cite{BramsonZeitouniZerner06} have a uniformly elliptic, stationary, totally ergodic, and even mixing example of an environment satisfying \eqref{notin}. 

In the present article, we construct an environment with similar properties for dimension $d=2$. Our main theorem is indeed: 
\begin{thm}\thlabel{mainth}
 For any $\varepsilon>0$, there is an $\varepsilon$--perturbative, stationary, mixing random environment $\omega=(\omega_u)_{u\in\Z^2}$ with associated probability measure $P$ such that for the associated random walk $((X_t),P_0^\omega)$, it holds that
\begin{equation}
 PP_0^\omega\big(X_t\cdot\vec1\xrightarrow[t\to\infty]{}\infty\big)>0 \text{ as well as }PP_0^\omega\big(X_t\cdot\vec1\xrightarrow[t\to\infty]{}-\infty\big)>0.
\end{equation}
Here, $\vec1$ denotes the vector $(1,1)$.
\end{thm}
A preprint by Guo \cite{guo} is concerned with the limiting velocity of the random walk in random environment on the events $\{X_t\cdot v\xrightarrow[t\to\infty]{}\dagger\infty\}$, $\dagger\in\{+,-\}$, $v\in\R^d$, in dimensions $d\ge2$, in the case where the random environment satisfies uniform ellipticity and a certain strong mixing condition which holds in Gibbsian environments, for instance.

\begin{proof}[Proof of \thref{mainth}, and organisation of the article]
 In Section \ref{randenv}, we construct an object called streetgrid which we use to define the actual random environment in Subsection \ref{transprob}. We prove the streetgrid to be stationary and mixing in the Subsections \ref{stationarity} and \ref{mixing}. These properties are inherited in the definition of the random environment. 
 
 In Subsection \ref{large}, we show that there are areas growing in the direction of $\vec1$ that are in some sense large. This has the consequence, via the placement of the transition probabilities, that the random walk has positive probability of never leaving these areas, while wandering off to infinity in the direction of $\vec1$. This is shown in Subsection \ref{largeindeed}. The same arguments could be repeated for $-\vec1$, which finishes the proof. 
\end{proof}

We should want to indicate some of the sources of inspiration that contributed to this article. The ideas of conducting the random walk to infinity on a ``treelike structure'' of ``not too slowly growing roads leading to infinity'' has been applied in \cite{BramsonZeitouniZerner06}. As for how to construct such a structure in dimension $d=2$, Häggström and Mester \cite{HaggstromMester09} had the idea of ever larger, ever rarer streets joining each other. By using Poisson processes of different intensities as the underlying structure instead of their``windows'' of fixed length, we were able to avoid some of the rigidity of their model and to make assertions on mixing, at the price of developing a completely new construction. 
\section{Construction of a random environment}\label{randenv}
\subsection{Notation}
\subsubsection{Boxes}
Recall the convention to write $u=(u_0,u_1)\in\Z^2$.
We call a \emph{box} any subset $B$ of $\Z^2$ that can be expressed as 
\begin{equation}
B=\{b_0,\dots,b_0'\}\times\{b_1,\dots,b_1'\}\text{ for some }b_j,b_j'\in\Z\text{ with }b_j\le b_j',\ j\in\{0,1\}.\label{boxfaces}
\end{equation}

For a box $B$, we define the \emph{emplacement of the faces of $B$} as
\begin{equation}
 b_j(B):=b_j,\quad b_j'(B):=b_j',\quad j\in\{0,1\}, \label{emplacements}
\end{equation}
where $b_j,b_j'$, $j\in\{0,1\}$, are taken from \eqref{boxfaces}.

For $v,w\in\Z^2$ we define the \emph{box between $v$ and $w$} as
\begin{equation}
\Btwn(v,w):=\Big\{\min\{v_0,w_0\},\dots,\max\{v_0,w_0\}\Big\}\times\Big\{\min\{v_1,w_1\},\dots,\max\{v_1,w_1\}\Big\}.\label{boxbetween}
\end{equation}

The (outer) \emph{boundary of a box $B$} may be defined as 
\begin{equation}
\partial B:=\{u\in\Z^2:d(u,B)=1\}; \label{boundary}
\end{equation}
here, $d(\cdot,\cdot)$ means the 1-metric. It is convenient to define as well the \emph{closure of $B$}, which is
\begin{equation}
 \closure B:=B\cup\partial B;\label{closure}
\end{equation}
the \emph{upper right corner $\UR B$} of a box $B$ is
\begin{equation}
\UR B:=(b_0'(B),b_1'(B)).\label{upperrightcorner}
\end{equation}
\subsubsection{Streets and streetgrid, and blocks}
We call a number $m\in\Nnull$ a \emph{superlevel}, and $k\in\{0,1\}$ a \emph{sublevel}. The mapping $(m,k)\mapsto2m+k:\ \Nnull\times \{0,1\}\rightarrow\Nnull$ is bijective, and this number is called the corresponding \emph{level}. Given any level $l\in\Nnull$, we can obviously reconstitute superlevel and sublevel using the inverse function, $(\RL{l},l\mod 2)$.

If a level has somehow been assigned to some object, we will speak of the superlevel and the sublevel of that object as well.

Given a level $l\in\Nnull$ and a function $F\in\Nnull^\mathbbm D$, $\mathbbm D\subseteq\Z^2$, a box $B\subseteq\mathbbm D$ is called a \emph{street of level $l$ w.r.t.\ $F$} if
\begin{equation}
F_u=l\textnormal{ for all }u\in B\textnormal{, and }F_u\neq l\text{ for all }u\in\partial B\cap\mathbbm D.\label{fieldstreet}
\end{equation}
We call it a \emph{field w.r.t.\ $F$} if it is a street of level $0$ w.r.t.\ $F$. When it is obvious or not important which level and function are meant, we will simply speak of ``street'' and ``field''.

We say $F$ is a \emph{streetgrid} if $\mathbbm D$ is the union of streets and fields with respect to $F$, i.e.
\begin{equation}
\mathbbm D=\bigcupdot_{l\in\Nnull}\smashoperator[r]{\bigcupdot_{\substack{B\text{ street of }\\\text{level }l\text{ w.r.t.\ }F}}}B.\label{streetgrid}
\end{equation}

Given a box $B$ contained in the domain of a streetgrid $F$, we define the \emph{level of the box $B$ w.r.t.\ $F$} as 
\begin{equation}
 \levelof(B)=\levelof^F(B):=\max_{u\in B}F_u.\label{levelofbox}
\end{equation}
Note that if the box $B$ is a street, the two definitions of ``level of the box $B$'' and ``level of the street $B$'' coincide.

For $B\subseteq\Z^2$ a box such that $\closure B\subseteq\mathbbm D$ the domain of $F$, we say \emph{$B$ is a block w.r.t.\ $F$} if all points $u\in\partial B$ are elements of exactly four different streets w.r.t.\ $F$, which are all of level greater than $\levelof^F(B)$. 

The \emph{upper and lower levels of the block $B$} are defined respectively as 
\begin{equation}
 \loben^F(B):=\max_{u\in\partial B}F_u\quad\text{ and }\quad\lunten^F(B):=\min_{u\in\partial B}F_u.\label{upperlower}
\end{equation}
$\lunten^\cdot(B)$ will be crucial in determining streets of which levels might be present if we have only information about $\partial B$ the boundary of $B$, and $\loben^\cdot(B)$ will constitute a lower bound to all levels that are not present in $\closure B$. 

Given a streetgrid $F$ and $u\in\Z^2$, we define $\Srnd_u^F$ to be the \emph{street or field around $u$}; to be precise, 
\begin{equation}
 \Srnd_u^F\text{ is defined to be the unique street or field }B\text{ w.r.t.\ }F\text{ such that }u\in B.
\end{equation}
\subsection{Construction of the streetgrid}
\subsubsection{Parameters and randomness used in the construction}
\begin{equation}
 \ROO_m:=(m+1)!^{-2}\text{ and }\plannedwidth_m:=m!^2,\ m\in\Nnull,\label{occurrenceandwidth}
\end{equation}
are called the \emph{rate of occurrence of streets at superlevel $m$} and the \emph{planned widths of the streets at superlevel $m$}, respectively.

We define 
\begin{equation}
 \calZ:=\Z\times\Nnull\times\Z^2,
\end{equation}
and, on some appropriate probability space $(\Omega,\filt F,P)$, a family of independent random variables 
\begin{equation}
\mathbf X:=\big(\mathbf X(x,l,w)\big)_{(x,l,w)\in\calZ},\label{X}
\end{equation}
which are to be Bernoulli-distributed with parameters $\ROO_{\RL{l}}$.

To understand the meaning of the index-set $\calZ$, we need to read it backwards. Every point in $\Z^2$ gets for every level in $\Nnull$ a Bernoulli-process $\{0,1\}^\Z$.

Having the necessary terms and definitions as well as the random ingredients at hand, we can start constructing the environment, beginning with a streetgrid. This will be done in two steps. Starting at the origin, we begin with narrow streets and make our way towards infinity by ever wider ones. This leaves wide areas of fields that will then be filled in the opposite direction with ever narrowing streets. 
\subsubsection{The initial grid}
We could put the random ingredient $\mathbf X$ directly into our construction, which will be built gradually in several definitions. We prefer however to write down these definitions as functions on $\{0,1\}^\calZ$, and to finally evaluate them at the random place $\mathbf X$. Notationwise, we will drop the dependence on $\mathbf x\in\{0,1\}^\calZ$ after the first appearence, though. Please note that the definitions may, for some $\mathbf x\in\{0,1\}^\calZ$, not make any sense; whenever there is some doubt on how $\mathbf x$ should look like, any typical realization $\mathbf x$ of $\mathbf X$ will do.

In a first step, we define processes that, roughly speaking, show where streets would be if each coordinate existed on its own. For each coordinate direction $j\in\{0,1\}$ and every superlevel $m\in\Nnull$, we attach to the left of every point highlighted as $1$ by the process $\mathbf x(\cdot,2m+j,0)$ an interval with the width $\plannedwidth_m$ of the respective superlevel $m$. Then, for any point, we take the maximum level of all streets the point lies in; that is, in the case of overlapping intervals of different levels, the higher level prevails:
{\allowdisplaybreaks[0]
\begin{align}
\MoveEqLeft W_x^j(\mathbf x):=2\max\big\{m\in\Nnull:\exists y\in\Z: x\le y<x+\plannedwidth_m,\ \mathbf x(y,2m+j,0)=1\big\}+j,\label{W}\\
&\qquad\qquad\qquad\qquad\qquad\qquad\qquad\qquad\qquad\qquad\ j\in\{0,1\},\ x\in\Z,\ \mathbf x\in\{0,1\}^\calZ.
\end{align}
}
We need to make sure $W_x^j(\mathbf X)$ is $P$--a.s.\ finite for all $j\in\{0,1\}$ and all $x\in \Z$. For $m\in\N$, $x\in\Z$, it holds that  
 \begin{align}
  \MoveEqLeft P\big(\exists\ y\in\Z:\ 0\le y-x<\beta_m,\ \mathbf X(y,2m+j,0)=1\big)\\
  &= P\big(\#\{y\in\Z:\ 0\le y-x<\beta_m,\ \mathbf X(y,2m+j,0)=1\}\ge1\big)\\
  &\le P\big(\#\{y\in\Z:\ 0\le y-x<\beta_m,\ \mathbf X(y,2m+j,0)=1\}\big)\\
  &=\smashoperator[r]{\sum_{y=x}^{x+\beta_m-1}}P\big(\mathbf X(y,2m+j,0)=1\big)=\beta_m\lambda_m=\frac{m!^2}{(m+1)!^2}=\frac1{(m+1)^2}.
 \end{align}
 With the Borel--Cantelli--lemma, we conclude that there are $P$--a.s.\ only finitely many $m\in\N$ satisfying the condition of the maximum in \eqref{W}, which hence is $P$--a.s.\ finite.

The dependence on $\mathbf x$ will de dropped for the next few definitions, even though it of course persists. 

The function $W_x^j$ will be further transformed by removing the outer intervals of smaller value in
\begin{equation}
V_x^j:=W_x^j\indic_{W_x^j=(\max_{0\le y\le x}W_y^j\vee\max_{x\le y\le 0}W_y^j)},\ j\in\{0,1\},\ x\in\Z.\label{funnelW}
\end{equation}
Note that the maximum over an empty set is to be read as $-\infty$.

The transition from $W_\cdot^0$ to $V_\cdot^0$ is visualized in Figure \ref{WV}.
\begin{rmk}\thlabel{WMI}
 A monotonically increasing function $f:\Nnull\to\R$ satifies $f_x=\max_{0\le y\le x}f_y$, $x\in\Nnull$. $(V_x^j)_{x\in\Nnull}$, $j\in\{0,1\}$ are not monotonically increasing, but ``weakly monotonically increasing, seen from $0$'' in the sense that they still satisfy
\begin{equation}
 V_x^j\in\Big\{0,\max_{0\le y\le x}V_y^j\Big\},\ x\in\Nnull,\ j\in\{0,1\},
\end{equation}
and a similar assertion for negative $x$.
\end{rmk}
\begin{figure}
\begin{center}
\begin{tikzpicture}
  \draw[very thin,<->](-7,0)--(7,0);
  \foreach \yy in{0,1,2}\draw[very thin](-0.2,\yy)--node[anchor=north west]{$\yy$}(0.2,\yy);
  \draw[very thin,->](0,-0.3)--(0,2.5);
  \foreach \xx in{-5,5}\draw[very thin](\xx,-0.2)--node[anchor=north west]{$\xx0$}(\xx,0.2);
  \begin{scope}[yshift=-4cm]
    \draw[very thin,<->](-7,0)--(7,0);
    \foreach \yy in{0,1,2}\draw[very thin](-0.2,\yy)--node[anchor=north west]{$\yy$}(0.2,\yy);
    \draw[very thin,->](0,-0.3)--(0,2.5);
    \foreach \xx in{-5,5}\draw[very thin](\xx,-0.2)--node[anchor=north west]{$\xx0$}(\xx,0.2);
  \end{scope}
  \begin{scope}
  \clip(-6.8,-4.6)rectangle(6.8,3.1);
  \foreach \xx in {-5.5}\fill[black!45!white](\xx,0)rectangle+(3.6,3);
  \foreach \xx in {-5.5}\fill[black!45!white](\xx,-4)rectangle+(3.6,3);

  \foreach \xx in {-3.9,2.7,6.3}\fill[black!35!white](\xx,0)rectangle+(0.4,2);
  \foreach \xx in {2.7,6.3}\fill[black!35!white](\xx,-4)rectangle+(0.4,2);

  \foreach \xx in {-6.7,-6.6,-5.6,-5.3,-5.0,-4.5,-4.3,-4.0,-3.5,-3.4,-2.7,-2.6,-2.4,-2.2,-0.3,0.5,0.8,1.2,1.5,2.1,2.7,2.9,3.1,3.5,4.0,5.1,5.4,5.8,6.1,6.2,6.3}\fill[black!15!white](\xx,0)rectangle+(0.1,1);
  \foreach \xx in {-0.3,0.5,0.8,1.2,1.5,2.1}\fill[black!15!white](\xx,-4)rectangle+(0.1,1);

  \foreach \xx in {-5.5}\draw[very thick](\xx,3) -- +(3.6,0);
  \foreach \xx in {-5.5}\draw[very thick](\xx,-1) -- +(3.6,0);

  \foreach \xx in {2.7,6.3}\draw[very thick](\xx,2) -- +(0.4,0);
  \foreach \xx in {2.7,6.3}\draw[very thick](\xx,-2) -- +(0.4,0);

  \foreach \xx in {-6.7,-6.6,-5.6,-0.3,0.5,0.8,1.2,1.5,2.1,3.1,3.5,4.0,5.1,5.4,5.8,6.1,6.2}\draw[very thick](\xx,1) -- +(0.1,0);
  \foreach \xx in {-0.3,0.5,0.8,1.2,1.5,2.1}\draw[very thick](\xx,-3) -- +(0.1,0);

  \foreach \xx in {-6.7,-6.6,-5.6,-0.3,0.5,0.8,1.2,1.5,2.1,3.1,3.5,4.0,5.1,5.4,5.8,6.1,6.2}\draw[very thick](\xx,1) -- +(0.1,0);
  \foreach \xx in {-0.3,0.5,0.8,1.2,1.5,2.1}\draw[very thick](\xx,-3) -- +(0.1,0);

  \foreach \xa/\xe in {-6.8/-6.7,-6.5/-5.6,-1.9/-0.3,-0.2/0.5,0.6/0.8,0.9/1.2,1.3/1.5,1.6/2.1,2.2/2.7,3.2/3.5,3.6/4.0,4.1/5.1,5.2/5.4,5.5/5.8,5.9/6.1,6.7/6.8}\draw[very thick](\xa,0) -- (\xe,0);
  \foreach \xa/\xe in {-6.8/-5.5,-1.9/-0.3,-0.2/0.5,0.6/0.8,0.9/1.2,1.3/1.5,1.6/2.1,2.2/2.7,3.1/6.3,6.7/6.8}\draw[very thick](\xa,-4) -- (\xe,-4);
  \end{scope}
\end{tikzpicture}
\end{center}
\caption{Simulation of a realization of $W_x^0(\mathbf X)$ and $V_x^0$, $x\in\Z$, represented by the thick line. The rectangles indicate the intervals attached to the points highlighted by the Bernoulli-processes of different intensities. Although in this picture the domain of the two functions looks continuous, they are defined to have domain $\Z$.}\label{WV}
\end{figure}
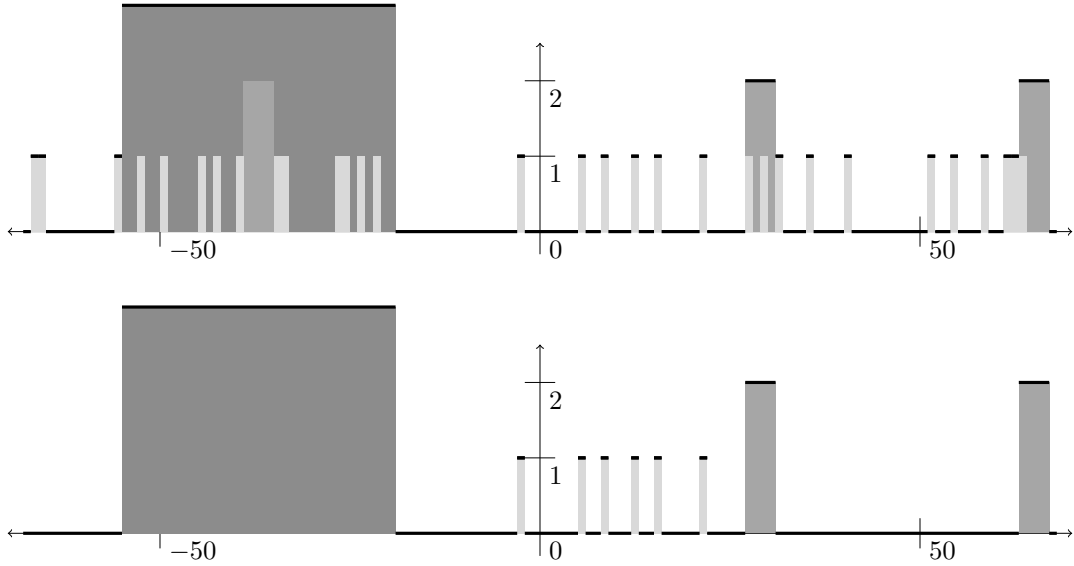
With the following definition, we begin our two--dimensional construction. Any point $u=(u_0,u_1)\in\Z^2$ gets assigned a level by 
\begin{equation}
 \IG^{\mathbf x}(u):=\big(V_{u_0}^0\vee V_{u_1}^1\big)\indic_{V_{u_0}^0\vee V_{u_1}^1\ge\max_{j\in\{0,1\}}\left(\max_{0\le x<u_j}V_x^j\vee\max_{u_j<x\le0}V_x^j\right)}.\label{funnelLu}
\end{equation}
In words, the point $u$ gets assigned the maximum of the two $V_{u_0}^0$ and $V_{u_1}^1$ provided this maximum is larger than any of the $V_x^j$ for $x$ between $0$ and $u_j$, with $j\in\{0,1\}$. Thus, $\IG$ satisfies a two-dimensional analogue of the heuristical notion of ``weakly monotonically increasing seen from $0$'' mentioned in \thref{WMI}. 

Note that $\IG$ is only the \emph{initial streetgrid}, and w.r.t.\ this $\IG$, large fields remain. 

We write $\IG(u):=\IG^{\mathbf X}(u)$; a simulation of $\IG$ is shown in Figure \ref{IG}.
\begin{figure}
\begin{center}
\begin{tikzpicture}
  \begin{scope}
    \clip(-6.8,-8)rectangle(6.8,8);
    \begin{scope}[black!75!white]
      \clip(-10.0,-10.0)rectangle(10.0,10.0);
      \foreach \x in {8.8} \fill(-10,\x)rectangle+(20,3.6);
      \end{scope}
    \begin{scope}[black!65!white]
      \clip(-10.0,-10.0)rectangle(10.0,8.8);
      \foreach \x in {-5.5} \fill(\x,-10)rectangle+(3.6,20);
      \end{scope}
    \begin{scope}[black!55!white]
      \clip(-1.9,-10.0)rectangle(10.0,8.8);
      \foreach \x in {-9.1,-6.4,-5.3,-2.6,2.5,3.4,3.8,6.3,7.5} \fill(-10,\x)rectangle+(20,0.4);
      \end{scope}
    \begin{scope}[black!45!white]
      \clip(-1.9,-2.2)rectangle(10.0,2.5);
      \foreach \x in {-3.9,2.7,6.3,9.3} \fill(\x,-10)rectangle+(0.4,20);
      \end{scope}
    \begin{scope}[black!35!white]
      \clip(-1.9,-2.2)rectangle(2.7,2.5);
      \foreach \x in {-9.6,-9.5,-9.4,-9.1,-7.8,-7.5,-5.8,-5.7,-5.4,-5.2,-5.1,-4.9,-4.7,-4.2,-4.0,-3.9,-3.1,-2.8,-2.5,-2.3,-1.7,-1.5,-1.3,-1.1,-0.8,-0.7,-0.6,-0.4,0.3,0.6,1.0,1.1,1.2,1.5,2.0,2.4,3.9,4.4,4.5,4.6,4.7,5.0,5.1,5.4,5.6,6.6,6.8,8.1,8.2,8.3,8.4,8.8,9.0,9.1,9.9} \fill(-10,\x)rectangle+(20,0.1);
      \end{scope}
    \begin{scope}[black!25!white]
      \clip(-1.9,-0.3)rectangle(2.7,0.3);
      \foreach \x in {-9.7,-9.5,-9.0,-8.8,-8.7,-7.9,-7.8,-7.2,-6.7,-6.6,-5.6,-5.3,-5.0,-4.5,-4.3,-4.0,-3.5,-3.4,-2.7,-2.6,-2.4,-2.2,-0.3,0.5,0.8,1.2,1.5,2.1,2.7,2.9,3.1,3.5,4.0,5.1,5.4,5.8,6.1,6.2,6.3,6.9,7.0,7.3,7.5,7.6,7.8,8.0,8.1,8.4,9.0,9.1,9.3,9.5,9.6} \fill(\x,-10)rectangle+(0.1,20);
      \end{scope}
    \fill[black!15!white](-0.2,-0.3)rectangle(0.5,0.3);
  \end{scope}
\end{tikzpicture}
\end{center}
\caption{Simulation of $\IG=\IG^{\mathbf X}$. Again, the domain of $\IG$ is not continuous, but $\Z^2$. $V_\cdot^1$ is the same as in Figure \ref{WV}.}\label{IG}
\end{figure}
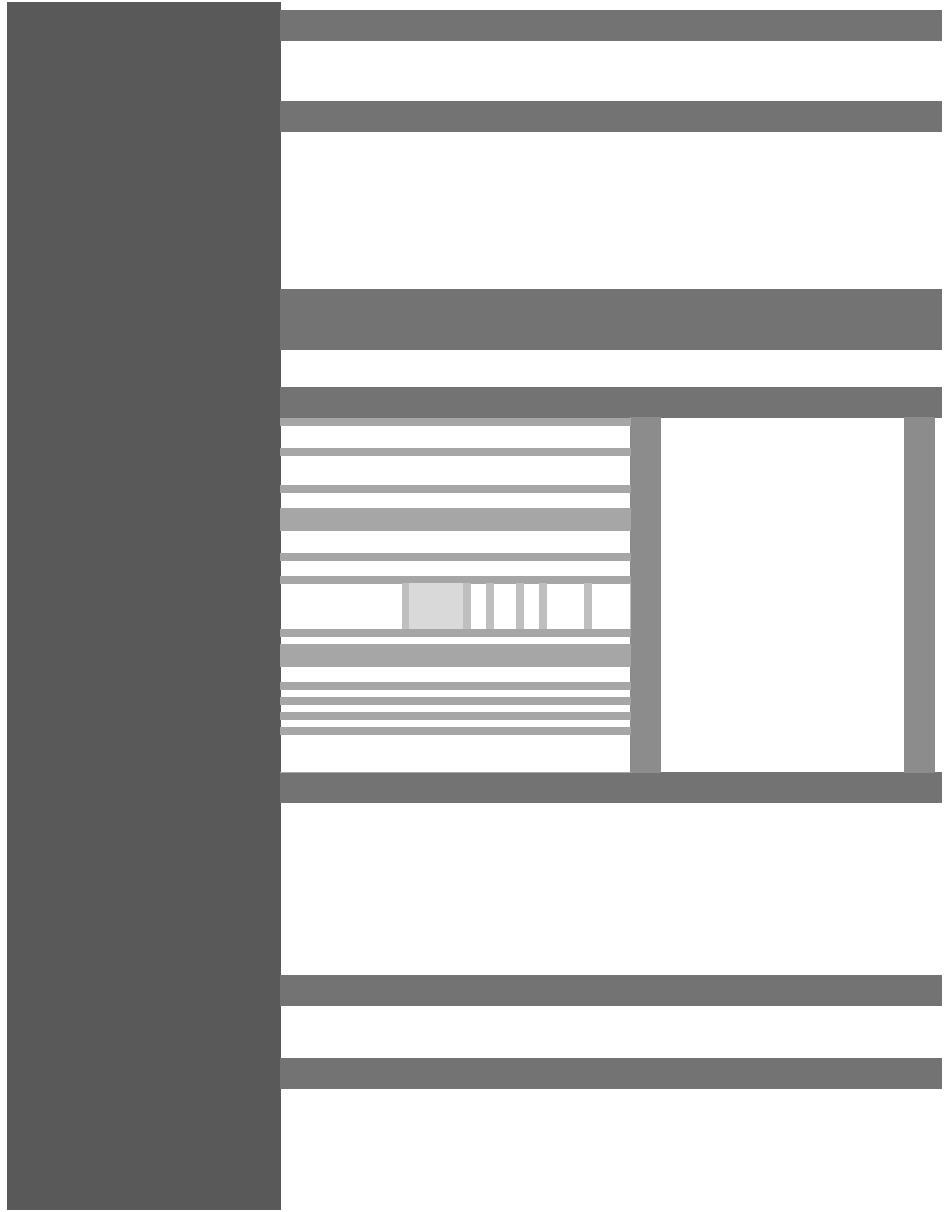
\begin{lem}\thlabel{InitgridStreetgrid}
 $\IG^{\mathbf X}(\cdot)$ is $P$--a.s.\ a streetgrid.
\end{lem}
\begin{proof}
 We need to show that $\Z^2$ is a patchwork of streets and fields w.r.t.\ $\IG^{\mathbf X}(\cdot)$ as in \eqref{streetgrid}. We will concentrate on the first quadrant, referring to analogy for the other ones.
 
 Define
 \begin{equation}
  Y_m^j:=\min\{x\in\Nnull|V_x^j> m\},\ m\in\Nnull,\ j\in\{0,1\}.\label{Ymj}
 \end{equation}
 On the coordinate axes, we have that
 \begin{align}
  \IG(0)&=\max_{j\in\{0,1\}}V_0^j,\\
  \IG(xe_0)&=\IG(0)\text{ for all }0\le x<Y_{\RL{\IG(0)}}^0,\\
  \IG(ye_1)&=\IG(0)\text{ for all }0\le y<Y_{\RL{\IG(0)}}^1,\\
  \IG(xe_0)&=V_x^0\text{ for all }x\ge Y_{\RL{\IG(0)}}^0,\\
  \IG(ye_1)&=V_y^1\text{ for all }y\ge Y_{\RL{\IG(0)}}^1.
 \end{align}
 On the first quadrant, it holds that 
 \begin{equation}
  \IG\big((x,y)\big)=\IG(0)\text{ for all }0\le x<Y_{\RL{\IG(0)}}^0,\ 0\le y<Y_{\RL{\IG(0)}}^1,
 \end{equation}
 and more generally, 
 \begin{equation}
  \IG\big((x,y)\big)=\begin{cases}
                     V_x^0&\text{if $V_x^0>V_z^1$ for all }0\le z\le y,\\
                     V_y^1&\text{if $V_y^1>V_z^0$ for all }0\le z\le x,\\
                     0&\text{else.}
                    \end{cases}
 \end{equation}
 If one takes this equation for fixed, say, $x$ with $V_x^0\neq0$ and lets run $y$ from $0$ to infinity, one gets the value $\IG((x,y))=V_x^0=\IG(xe_0)$ for all $y<\min\{y\in\Nnull|V_y^1>V_x^0\}$; in other words, until from the other coordinate, one gets blocked.  Because $V_\cdot^0$ and $V_\cdot^1$ have disjoint codomains (except for $0$, which they have in common), these blockings are sharp in the sense that one can always tell whether a point has got its value (different from $0$) from $V_\cdot^0$ or $V_\cdot^1$.

 Also, the other way around, if some point $(x,y)\in\Z^2$ has got its initial--grid--value from, say, $V_x^0$, then fixing $x$ and letting $z$ run from $y$ to $0$ yields
 \begin{equation}
  \IG\big((x,z)\big)=\IG\big((x,y)\big)\text{ for all }y\ge z\ge 0.
 \end{equation}
 Combining the arguments of the last two paragraphs, one can see that all points $u\in\Z^2$ satisfying $\IG(u)=2m+j\neq0$ lie in areas of constant $\IG$-value outgoing perpendicularily from the $j$-th coordinate axis. Each such area continues until it gets blocked by some area coming from the other coordinate axis. The areas are of rectangular shape, and $P$--a.s.\ finite. 

This applies as well to the areas where the initial grid equals $0$. These are indeed surrounded by four streets of different levels, so that they are fields. 
 
 Finally, we need not only to pay attention at the the four quadrants individually, but at the transition between them as well. Indeed, the streetgrid--property holds because between adjacent quadrants, the same $V_\cdot^j$, $j\in\{0,1\}$ influences the construction of the streets. 
\end{proof}
\begin{rmk}\thlabel{informalreponsibility}
 We will be saying ``\,$0$ is responsible in $\IG$ for the emplacement of streets of level $l$ on $D$'' for any block $D$ w.r.t.\ $\IG$ containing the origin and any level $\levelof^{\IG}(D)\le l<\lunten^{\IG}(D)$.
\end{rmk}
$\levelof^{\IG}(D)$ is the highest level of any streets placed on $D$. The  emplacement of these streets has been provided by the random ingredient $\mathbf X$ evaluated at points $(\cdot,l,0)$, so it is sound to say $0$ is responsible. 

How about the levels $\levelof^{\IG}(D)<l<\lunten^{\IG}(D)$? No street of these levels exists in $D$. But this absence of streets was stipulated by an absence of $1$s in the random ingredient $\mathbf X$ at points $(\cdot,l,0)$, $\levelof^{\IG}(D)<l<\lunten^{\IG}(D)$. So it is legitimate to say $0$ is responsible for those levels as well. 

We will extend the notion of responsibility in \thref{responsibility}. 
\subsubsection{Asphalting of the remaining fields}
After constructing $\IG^{\mathbf x}(u)$, we continue by iteratively putting the missing streets on the remaining fields. Let us describe informally how we proceed. 

The streets that are not fields w.r.t.\ $\IG^{\mathbf x}$ are to remain untouched. We want to work exclusively on the fields. 

By \thref{InitgridStreetgrid}, any field $B$ w.r.t.\ $\IG^{\mathbf x}$ is surrounded by four streets. The minimum of their level minus one is the level of the first streets that should be put on $B$. Determining the level of the streets to put is hence the first step. 

Then, we need to know the place where we put these streets. To each field $B$ will be assigned an own process resembling the one in \eqref{W}; this time however, only one level at a time is taken into account. The random ingredient of this process will be the Bernoulli process associated to the upper right corner of $B$ and the respective level.

Now, when the streets are put on the fields, smaller fields are created; on these, we put streets of the next lower level, and so on. 

Now, back to rigid definitions. First, we define a dummy and the starting point of the iteration,
\begin{equation}
 L_u^0:\equiv0,\quad L_u^1:=\IG^{\mathbf x}(u),\ u\in\Z^2.
\end{equation}
For $i\ge1$, and $B$ a field with respect to the $i$-th iteration step $L_\cdot^i$, we associate a level to $B$ by 
\begin{equation}
 l^i(B):=\begin{cases}
           \min_{v\in\partial B}L_v^i-1&\text{if $B$ is not a field with respect to }L_\cdot^{i-1}\\
           l^{i-1}(B)-1&\text{if it is.}
         \end{cases}
\end{equation}
This is the level of the streets that are going to be placed on $B$. The first line of the definition is used at the first iteration step, and also the default for the following steps; only if there has no street been put on a field in the last step, the second line makes sure that in the current step, the same level is not used again.

We provide the emplacement in $B$ for the new streets of level $l$ (we exceptionally remind the dependence on $\mathbf x$) by
\begin{equation}
 W_x^{l,B}(\mathbf x):=l\indic_{\exists y\in\Z:\ x\le y<x+\plannedwidth_{\RL{l}},\mathbf x(y,l,\UR B)=1},l\in\Nnull,\ x\in\Z.
\end{equation}
Given $l$, the indicator function checks whether at the point $x$, there is a street of level $l$ induced by the Bernoulli process at the upper right corner of the field. 

The streets are placed on the field $B$ using
\begin{equation}
 L_u^{l,B}:=W_{u_{\SL{l}}}^{l,B}\indic_{u\in B},\ u\in\Z^2.
\end{equation}
The sublevel $\SL{l}$ is taking care of the (vertical or horizontal) orientation of the streets. 

We need to do this setting of streets in every field, and set the whole iteration step as  
\begin{equation}
 L_u^i:=L_u^{i-1}+\smashoperator{\sum_{\substack{B\text{ field}\\\text{w.r.t.\ }L_\cdot^{i-1}}}}L_u^{l^i(B),B}. 
\end{equation}
We have put, on every field w.r.t.\ $L_\cdot^{i-1}$, streets of ``one level lower''.

The process $L_\cdot^i=L_\cdot^i(\mathbf x)$ converges pointwise with $i\rightarrow\infty$ for $P$--almost any realization $\mathbf x$ of $\mathbf X$: for any field w.r.t.\ $\IG^{\mathbf x}(\cdot)$, at some iteration, the level $1$ (with superlevel $0$) is reached and the remaining sub-fields are entirely filled with streets of level $1$. Another way of seeing the convergence is by remarking that for every point $u\in\Z^2$, the sequence $(L_u^i)_{i\in\N}$ is monotonically increasing and bounded. The limes will be called the \emph{final streetgrid} $\SG(\mathbf x)=(\SG(\mathbf x)_u)_{u\in\Z^2}$ and we write $\SG=(\SG_u)_{u\in\Z^2}:=(\SG(\mathbf X)_u)_{u\in\Z^2}$.

Based on the earlier simulation of the initial grid, a simulation of the final streetgrid can be found in Figure \ref{SG}.
\input{FigureC}
\begin{lem}
 $\SG(\mathbf X)$ is $P$--a.s.\ a streetgrid.
\end{lem}
\begin{proof}
 Each iteration step $L^i$ is: to obtain $L^i$, only the fields of $L^{i-1}$ are changed, and on these fields are placed streets extending in one coordinate direction up to the boundary of the field they are placed on. These streets are of strictly lower level than all surrounding streets. 

 Because the passage to the limit is of the type where for any finite region, the sequence is from some point on constant and  equal to the limiting object, $\SG$ is a streetgrid as well. 
\end{proof}
We turn again towards the concept of responsibility. This time, we give a precise definition, and then explain how it relates to our construction of the streetgrid.
\begin{dfn}\thlabel{responsibility}
 Take a streetgrid $g$. For any block $D$ w.r.t.\ $g$ and any $\levelof^g(D)\le l<\lunten^g(D)$, there is a unique $w\in\Z^2$ of which we say that it is \emph{in $g$ responsible for the emplacement of streets of level $l$ in $D$}. It is given by $w=0$ if $0\in D$, and $w=\UR D$ if $0\not\in D$.  
\end{dfn}
For $g=\IG$, this definition exactly reflects \thref{informalreponsibility}. The streets already present in $\IG$ are carried over to $\SG$, so it is reasonable to say $0$ is responsible for these in $\SG$ as well. 

The responsibility of points $w\neq0$ can be understood as follows: Any field $D$ w.r.t.\ $\IG$ does not contain the origin. It is also a block and will remain a block in the course of the construction. 

The first iteration step is about placing streets of level $l=\lunten^{\IG}(D)-1$ on $D$. The randomness for their emplacement comes from the Bernoulli process $\mathbf X((\cdot,l,\UR D))$. This is why $\UR D$ should be considered responsible for this block and level. 

If no streets of level $l$ are placed (because the Bernoulli process is $0$ in the relevant range), $\UR D$ is responsible for the subsequent lower levels as well, until streets is placed. The level of these streets will later turn out to be the level $\levelof^{\SG}(D)$ of the block $D$. 

By the placement of these streets, smaller fields are created, and it is \emph{their} upper right corner that provides the randomness via $\mathbf X$. These upper right corners are hence the places that are responsible for the streets of these lower levels, on these smaller fields (which again are and remain blocks). 

The next level shows that there is no conflict of responsibility.
\begin{lem}\thlabel{responsibleforoneblock}
 Take a streetgrid $g$. If $w\in\Z^2$ is responsible in $g$ for the emplacement of streets of level $l$ in $D$, where $l\in\N$ is a level and $D$ some block w.r.t.\ $g$, then $w$ is \emph{not} responsible in $g$ for the emplacement of streets of level $l$ in $B$, where $B\neq D$ is some other block w.r.t.\ $g$.
\end{lem}
\begin{proof}
 Suppose $w$ is responsible in $g$ for the emplacement of streets of level $l$ in both $D$ and $B$, where both $D$ and $B$ are blocks w.r.t.\ $g$, but $D\neq B$. A first deduction is that either both $B$ and $D$ must contain the origin, or share the same upper right corner $w=\UR B=\UR D$. In either case, $B\cap D\neq\emptyset$. 

 As $B\neq D$, this implies that, without loss of generality, $\partial B\cap D\neq\emptyset$. Hence, $\levelof^g(D)\ge\lunten^g(B)$. This is a contradiction to that $w$ was to be responsible for the same level in $B$ and $D$.  
\end{proof}
\subsection{Transition probabilities for the random environment}\label{transprob}
In order to determine where what transition kernels will be placed, we cut down the streets of the streetgrid to lanes using the following definition:
\begin{dfn}
For $\diamondsuit,\heartsuit\in\{+,-\}$, $B$ a street w.r.t.\ $\SG(\mathbf X)$ of superlevel $m:=\RL{\levelof^{\SG}(B)}\ge2$ and sublevel $k:=\SL{\levelof^{\SG}(B)}$, we define the \emph{lanes}
\begin{equation}
 \Lane_{\diamondsuit,\heartsuit}^{\SG}(B):=\begin{cases}
                             \{u\in B:b_k(B)\phantom{+\frac{\plannedwidth_m}4\ }\le u_k<b_k(B)+\frac{\plannedwidth_m}4\}&\text{if } \diamondsuit=+,\heartsuit=+;\\
                             \{u\in B:b_k(B)+\frac{\plannedwidth_m}4\le u_k<b_k(B)+\frac{\plannedwidth_m}2\}&\text{if } \diamondsuit=+,\heartsuit=-;\\
                             \{u\in B:b_k'(B)-\frac{\plannedwidth_m}2<u_k\le b_k'(B)-\frac{\plannedwidth_m}4\}&\text{if } \diamondsuit=-,\heartsuit=-;\\
                             \{u\in B:b_k'(B)-\frac{\plannedwidth_m}4<u_k\le b_k'(B)\}&\text{if } \diamondsuit=-,\heartsuit=+.
                            \end{cases}
\end{equation}
The definition of $b_\cdot(B)$ and $b_\cdot'(B)$ was given in \eqref{emplacements}.
\end{dfn}
Note that there might be some non--empty space between the two middle lanes $\Lane_{+,-}^{\SG}(B)$ and $\Lane_{-,-}^{\SG}(B)$.

We want to place the transition probabilities in a way that on the lanes with ``$+$'' as first index, the random walk feels a drift northwards or eastwards (if the sublevel of the street is $0$ or $1$, respectively), and on the lanes with ``$-$'' as first index, it feels a drift to the south or the west. The distinction between $+$ and $-$ in the second index is then used to provide a drift to the area where two lanes of the same street with the same first index meet. 

\begin{dfn}\thlabel{environment}
With $\diamondsuit,\heartsuit\in\{+,-\}$, we define 
 \begin{align}
  \omega_{\diamondsuit,\heartsuit}&:\{\pm e_0,\pm e_1\}\rightarrow[\frac14-\varepsilon,\frac14+\varepsilon],\\
  \omega_{\diamondsuit,\heartsuit}(\dagger e_0)&:=\frac14+(\dagger(\diamondsuit(\heartsuit \varepsilon))),\\
  \omega_{\diamondsuit,\heartsuit}(\dagger e_1)&:=\frac14+(\dagger(\diamondsuit \varepsilon)),\ \dagger\in\{+,-\},\\
\shortintertext{and}
  \omega_{\frac14}(e)&:=\frac14,\ e\in\{\pm e_0,\pm e_1\}.
 \end{align}
\end{dfn}
We will also be using the notation $\omega_\nearrow=\omega_{+,+}$ and visualize this local transition probability either by 
\begin{tikzpicture}
 \def\uebergang{\filldraw[fill=black](0,0)circle(0.02cm);\draw[<->](-0.1,0)--(0.2,0);\draw[<->](0,-0.1)--(0,0.2);}
 \uebergang
\end{tikzpicture} or $\nearrow$. See also Figure \ref{transitionsfigure}.

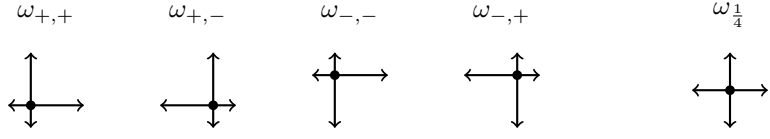
\begin{figure}[tbh]
\begin{center}
\begin{tikzpicture}[thick]
 \def\uebergang{\filldraw[fill=black](-0.2,-0.2)circle(0.05cm);\draw[<->](-0.5,-0.2)--(0.5,-0.2);\draw[<->](-0.2,-0.5)--(-0.2,0.5);}
 \foreach\xx/\aa/\bb/\rota in {0/+/+/0,1/+/-/90,2/-/-/270,3/-/+/180}
 {
  \draw(\xx*2,1)node{$\omega_{\aa,\bb}$};
  \begin{scope}[shift={(\xx*2,0)},rotate=\rota]\uebergang\end{scope}
 }
 \begin{scope}[xshift=9cm]
  \draw (0,1)node{$\omega_{\frac14}$};
  \filldraw[fill=black](0,0)circle(0.05cm);
  \draw[<->](-0.5,0)--(0.5,0);
  \draw[<->](0,-0.5)--(0,0.5);
 \end{scope}
 \end{tikzpicture}
 \caption{The transition probability kernels $\omega_{\cdot,\cdot}^j$. The lengths of the arrows are not to scale.}\label{transitionsfigure}
 \end{center}
\end{figure}
We will need a reflection matrix, namely
\begin{equation}
 R:=\left(\begin{array}{cc}0&1\\1&0\end{array}\right),
\end{equation}
to place the transition probabilities we just defined on the streets. 
\begin{dfn}
Given the streetgrid $\SG(\mathbf X)$, the transition probability kernels of the environment at place $u\in\Z^2$ will be defined as follows. If $u\in B$ a street w.r.t.\ $\SG(\mathbf X)$ such that $\RL{\levelof(B)}\ge2$ and $b'_{\SL{\levelof(B)}}(B)-b_{\SL{\levelof(B)}}(B)+1\ge\plannedwidth_{\RL{\levelof(B)}}$, set
\begin{equation}
 \omega_u=\omega_u(\mathbf X):=\begin{cases}
            \omega_{\diamondsuit,\heartsuit}&\text{if }u\in \Lane_{\diamondsuit,\heartsuit}^{\SG}(B),\ \diamondsuit,\heartsuit\in\{+,-\},\ \SL{\levelof(B)}=0,\\
            \omega_{\diamondsuit,\heartsuit}\circ R&\text{if }u\in \Lane_{\diamondsuit,\heartsuit}^{\SG}(B),\ \diamondsuit,\heartsuit\in\{+,-\},\ \SL{\levelof(B)}=1,\\
            \omega_{\frac14},&\text{else.}
           \end{cases}
\end{equation}
If $u\in B$ any other street, set $\omega_u:=\omega_{\frac14}$.

Here, $\omega_{\diamondsuit,\heartsuit}\circ R(e)=\omega_{\diamondsuit,\heartsuit}(Re)$, $e\in\{\pm e_0,\pm e_1\}$.
\end{dfn}
A visualization of the lanes and the different corresponding transition probabilities can be found in Figure \ref{lanesfigure}.

\def\uebergang[#1]#2{\begin{tikzpicture}[rotate=#1,xscale=#2, scale=0.7]\draw[<->](-0.19,0)--(0.75,0);\draw[<->](0,-0.19)--(0,0.75); \filldraw[fill=black](0,0)circle(0.05cm);\end{tikzpicture}}
\begin{figure}[htb]
\begin{center}
\begin{tikzpicture}
\draw[white] (1.5,0)-- node[anchor=east,xshift=-0.3cm,black]{$\Lane_{+,+}^{\SG}(B)$} node[anchor=west,black]{\uebergang[270]{-1}}
      (1.5,1)-- node[anchor=east,xshift=-0.3cm,black]{$\Lane_{+,-}^{\SG}(B)$} node[anchor=west,black]{\uebergang[180]{-1}}
      (1.5,2)--                                     node[anchor=west,black]{\begin{tikzpicture}[scale=0.7]\filldraw[fill=black](1,1)circle(0.05cm);\draw[<->]   (0.5,1)--(1.5,1);\draw[<->](1,0.5)--(1,1.5);\end{tikzpicture}}
      (1.5,3.5)--  node[anchor=east,xshift=-0.3cm,black]{$\Lane_{-,-}^{\SG}(B)$} node[anchor=west,black]{\uebergang[0]{-1}}
      (1.5,4.5)--  node[anchor=east,xshift=-0.3cm,black]{$\Lane_{-,+}^{\SG}(B)$} node[anchor=west,black]{\uebergang[90]{-1}} (1.5,5.5);
\draw[white] (3,0) -- node[anchor=south,black]{$\Lane_{+,+}^{\SG}(B')$}node[anchor=south,yshift=2cm,black]{\uebergang[0]{1}}
                    (5.5,0) --node[anchor=south,black]{$\Lane_{+,-}^{\SG}(B')$}node[anchor=south,yshift=2cm,black]{\uebergang[90]{1}}
                      (8,0) -- node[anchor=south,black]{$\Lane_{-,-}^{\SG}(B')$}node[anchor=south,yshift=2cm,black]{\uebergang[270]{1}}
                    (10.5,0) -- node[anchor=south,black]{$\Lane_{-,+}^{\SG}(B')$}node[anchor=south,yshift=2cm,black]{\uebergang[180]{1}} (13,0);
\foreach \shifter in {0cm,5.5cm}\draw[yshift=\shifter,thick](-1,0)--(3,0);
\foreach \shifter in {2cm,3.5cm}\draw[yshift=\shifter, very thin](-1,0)--(3,0);
\foreach \shifter in {1cm,4.5cm}\draw[yshift=\shifter,dashed] (-1,0)--(3,0);
\foreach \shifter in {0cm,10cm}\draw[xshift=\shifter,thick](3,-0.5)--(3,6);
\foreach \shifter in {5cm}\draw[xshift=\shifter,very thin](3,-0.5)--(3,6);
\foreach \shifter in {2.5cm,7.5cm}\draw[xshift=\shifter,dashed](3,-0.5)--(3,6);
\draw(0,5)node[anchor=south,yshift=0.3cm,fill=white]{Sublevel $1$};
\draw(8,5)node[anchor=south,yshift=0.3cm,fill=white]{Sublevel $0$};
\draw(-1,0)node[fill=white]{$B$};
\draw(13,0)node[fill=white]{$B'$};
\end{tikzpicture}
 \end{center}
 \caption{The horizontal street $B$ (of sublevel $1$) joining the vertical street $B'$ (of sublevel $0$). The street $B$ to the left is wider than its planned width, so that there is some space between the lanes $\Lane_{+,-}^{\SG}(B)$ and $\Lane_{-,-}^{\SG}(B)$. The width of the streets is not to scale: $B'$ ought to be much wider.}\label{lanesfigure}
\end{figure}
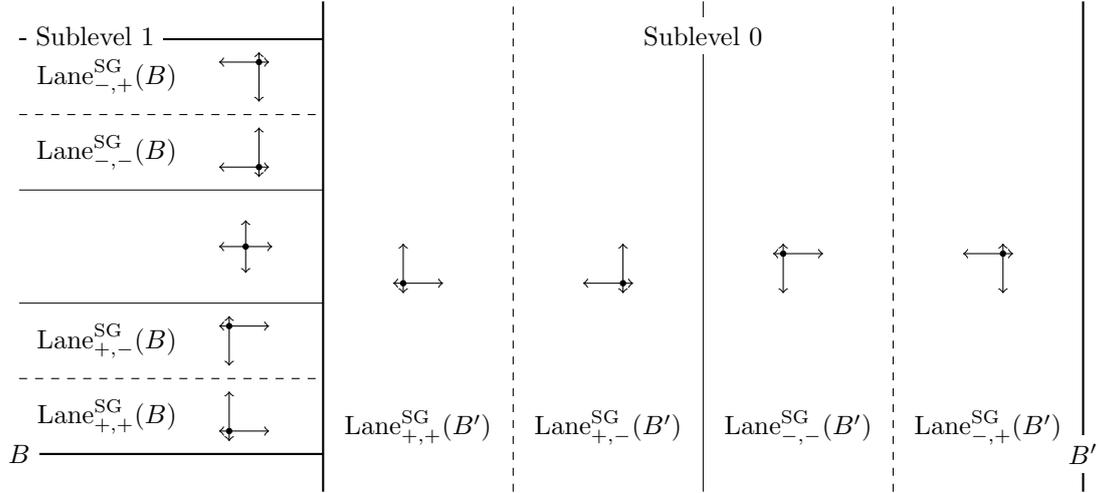
\section{Properties of $\IG$ and $\SG$}
\subsection{Heuristial approach}
 Let us describe a very simple model of a random walk in a non--random environment. Define the environment $\varpi$ by setting
 \begin{align}
  \varpi_u&\begin{cases}\omega_\searrow&\text{for all }u\in\Z^2\text{ such that }u_1\ge0,\\
                        \omega_\nearrow&\text{for all }u\in\Z^2\text{ such that }u_1<0.
               \end{cases}
 \end{align}
 That is, the random walk is subject to a uniform drift in direction of $e_0$ and towards the zeroth coordinate axis. It is easy to prove by standard martingale methods and the Borel--Cantelli--Lemma that the associated random walk in random environment $(X_n)_n$ starting at $0$ has positive probability never to leave the set $\{x\in\Z^2|x_0\ge0,|x_1|\le\sqrt{x_0}\}$, while following the first coordinate axis to infinity. 

 The morality of this example is that a random walk with uniform drift along a line and with a drift pushing it back towards that line has positive probability to never be further away from the line than the square root of the travelled length. 

 We will prove that $P$--almost surely, somewhere, there is a street w.r.t.\ $\IG^{\mathbf X}$ on the first coordinate axis satisfying the following: if one walks down that street (northwards) until one hits a perpendicular street, walks eastwards on that new one until the next perpendicular street, starts walking northwards again, and so on; if one does so, then:
 \begin{itemize}
  \item at the end of one street, one always encounters one of the next higher level;
  \item the width of these streets grows nicely, 
  \item the streets are not too long. 
 \end{itemize}
 Also, there will always be a drift pushing forward and to the middle of the two lanes with first index ``+'' in these streets. 

 The idea is that, when walking like described above, the width of the street the walker is in as a function of the distance travelled is larger than the square root $(\cdot)^{1/2}$; this is in analogy to the above example.

 An average--case--analysis shows heuristically why this is the case. 

 The streets of superlevel $m$ have a planned width of $\plannedwidth_m$ and, on average, a length of less than $\frac1{\ROO_{m+1}}$. The somewhat worst case for the random walk is if it has to go through the whole length of every street. The width of the $n$--th street the random walk visits is $\plannedwidth_n=n!^2$. The distance travelled is of the order of 
 \begin{equation}
  \sum_{i=1}^n\frac1{\ROO_{i+1}}=\sum_{i=1}^n(i+2)!^2\le2(n+2)!^2.
 \end{equation}
 This shows that the square root of the travelled distance is of slower growth than the width of the streets, leaving enough room to the random walk for fluctuations without leaving the sequence of streets. 
  
 The exact proof stretches over the whole subsection, but the most pertinent statements can be found in \thref{corev,widthandlength,evstreetsjoin}.
 \begin{rmk}
  The statements above are even true with \emph{any} root $(\cdot)^{1/\alpha}$, $\alpha>1$, instead of the square root $(\cdot)^{1/2}$. Hence we need to fix the exponent. 
 \end{rmk}
 \begin{dfn}
  \begin{equation}
   \text{Set $\alpha>1$ for the rest of the article.}\label{alpha}
  \end{equation}
 \end{dfn}
\subsection{The way to infinity is eventually large}\label{large}
 We start the proof of the above claims with a seemingly technical \thnameref{G} and \thnameref{threeevents}.

 \begin{dfn}\thlabel{G}
  The point inducing the (lowest part of the) first street of superlevel $m\in\Nnull$ on the $j$-th coordinate axis is defined as 
  \begin{equation}
   G_m^j:=\min\{x\in\Nnull:\mathbf X(x,2m+j,0)=1\},\ j\in\{0,1\}.\label{Gmj}
  \end{equation}
 \end{dfn}
 In view of the following Lemma, let us recall from \eqref{occurrenceandwidth} the parameters $\ROO_m:=(m+1)!^{-2}$ and $\plannedwidth_m:=m!^2,\ m\in\Nnull$.
 \begin{lem}\thlabel{threeevents}
  The following events all happen $P$--almost surely only finitely often (in $m$): 
  \begin{gather}
   \{G_{m-1}^0\ge G_m^0-\plannedwidth_m+1\};\quad\{G_{m-1}^1-\plannedwidth_{m-1}+\plannedwidth_m\ge G_m^1-\plannedwidth_m+1\};\\
   \{G_m^j>(\ROO_m)^{-\alpha}\},\ j\in\{0,1\}.
  \end{gather}
 \end{lem}
 \begin{proof}
  The $G_m^j$, $m\in\Nnull,$ $j\in\{0,1\}$, are geometrically distributed, independent random variables with success probability $\ROO_m=\frac1{(m+1)!^2}$. We can calculate, for the first event, 
  \begin{align}
   \MoveEqLeft P(G_{m-1}^0\ge G_m^0-\plannedwidth_m+1)\\
   &=\sum_{x\in\Nnull}P(G_{m-1}^0\ge x-\plannedwidth_m+1)P(G_m^0=x)\\
   &=\sum_{x\in\Nnull}(1-\ROO_{m-1})^{x-\beta_m+1}(1-\ROO_m)^x\ROO_m\\
   &=(1-\ROO_{m-1})^{-\plannedwidth_m+1}\ROO_m\sum_{x\in\Nnull}[(1-\ROO_{m-1})(1-\ROO_m)]^x\\
   &=(1-\ROO_{m-1})^{-\plannedwidth_m+1}\frac{\ROO_m}{\ROO_{m-1}+\ROO_m-\ROO_{m-1}\ROO_m}\\
   &=(1-\ROO_{m-1})^{-\plannedwidth_m+1}\Big(\frac{\ROO_{m-1}}{\ROO_m}+1-\ROO_{m-1}\Big)^{-1}\\
   &=(1-\ROO_{m-1})^{-\plannedwidth_m+1}\big((m+1)^2+1-\ROO_{m-1}\big)^{-1},\ m\in\N.
  \end{align}
 We see that the first term converges, while the second one is summable, so that we can conclude using the Borel-Cantelli-Lemma.

 The probability of the second event computes just the same way, only the limit of the leading term is some other constant. 

 For the last event, we observe that 
 \begin{equation}
  P\big(G_m^j>(\ROO_m)^{-\alpha}\big)=(1-\ROO_m)^{\FF{\ROO_m^{-\alpha}+1}}=\Big[\Big(1-\frac1{(m+1)!^2}\Big)^{(m+1)!^2}\Big]^{\frac{\FF{(m+1)!^{2\alpha}}+1}{(m+1)!^{2}}}\sim e^{-(m+1)!^{2\alpha-2}}
 \end{equation}
 is indeed summable as well.
\end{proof}
\begin{cor}\thlabel{corev}
 The event 
 \begin{equation}
  \big\{G_{m-1}^0+\plannedwidth_m\le G_m^0\le(\ROO_m)^{-\alpha}\big\}\cap\big\{G_{m-1}^1+2\plannedwidth_m-\plannedwidth_{m-1}\le G_m^1\le(\ROO_m)^{-\alpha}\big\}\label{eventually}
 \end{equation}
 holds $P$--a.s.\ eventually. Hence, $P$--a.s.,
 \begin{equation}
  M(\mathbf X):=\min\bigg\{m'\ge5\big|\omega\in\smashoperator[l]{\bigcap_{m=m'}^\infty}\splitfrac{\{G_{m-1}^0+\plannedwidth_m\le G_m^0\le\ROO_m^{-\alpha}\}}{\cap\{G_{m-1}^1+2\plannedwidth_m-\plannedwidth_{m-1}\le G_m^1\le\ROO_m^{-\alpha}\}}\bigg\}<\infty.\label{dfnM}
 \end{equation}
\end{cor}
$M=M(\mathbf X)$ is the superlevel from which the event defined in \eqref{eventually} always holds. The restriction to $m'\ge5$ is made so that we do not have to worry about whether we can divide streets into four lanes, and subdivide lanes in four equal parts: already $\plannedwidth_4=576=36*16$.

There is a picture relating the terms of the event \eqref{eventually} in Figure \ref{eventuallyfigure}.
\begin{figure}[htb]
 \begin{center} 
  \begin{tikzpicture}
 \begin{scope}
  \clip(-0.3,-0.3)rectangle(9.5,8.5);
  \filldraw[black!10!white](-0.3,-0.3)rectangle(3,1);
  \filldraw[black!20!white](3,-0.3)rectangle(4,1);
  \filldraw[black!30!white](-0.3,1)rectangle(6,2);
  \filldraw[black!40!white](6,-0.3)rectangle(9,5);
  \filldraw[black!50!white](-0.3,5)rectangle(10,8);
  \foreach\xx in {3,4}\draw(\xx,-0.3)--(\xx,1);
  \foreach\xx in {6,9}\draw(\xx,-0.3)--(\xx,5);
  \foreach\yy in {1,2}\draw(-0.3,\yy)--(6,\yy);
  \foreach\yy in {5,8}\draw(-0.3,\yy)--(10,\yy);
 \end{scope}
 \draw[very thin,<->](-0.3,0)--(9.5,0);
 \draw[very thin,<->](0,-0.3)--(0,8.5);
 \draw(0,0)node[anchor=north east]{$0$};
 \draw(4,0)node[anchor=north,yshift=-0.3cm]{$G_{m-1}^0$};
 \draw(8,0)node[anchor=north,yshift=-0.3cm]{$G_m^0$};
 \draw(0,2)node[anchor=east,xshift=-0.3cm]{$G_{m-1}^1$};
 \draw(0,7)node[anchor=east,xshift=-0.3cm]{$G_m^1$};
 \draw[thick,decorate,decoration={brace,mirror,raise=1pt}](4,1)--node[anchor=north,below=3pt]{$\ge0$}(6,1);%
 \draw[thick,decorate,decoration={brace,raise=1pt}](6,4)--node[anchor=east,left=3pt]{$\ge0$}(6,5);%
 \draw[thick,decorate,decoration={brace,mirror,raise=1pt}](0,-1)--node[anchor=north,below=3pt]{$\le(\ROO_m)^{-\alpha}$}(9,-1);%
 \draw[thick,decorate,decoration={brace,raise=1pt}](-1.5,0)--node[anchor=east,left=3pt]{$\le(\ROO_m)^{-\alpha}$}(-1.5,8);%
 \draw[|<->|](6,0.3)--node[anchor=south]{$\plannedwidth_m$}(9,0.3);
 \draw[|<->|](6.1,1)--node[anchor=west]{$\plannedwidth_m$}(6.1,4);
 \draw[|<->|](5.7,1)--node[anchor=east]{$\plannedwidth_{m-1}$}(5.7,2);
 \draw[|<->|](3,5)--node[anchor=east]{$\plannedwidth_m$}(3,8);
\end{tikzpicture}
 \end{center}
\caption{The implications of the event in \protect\eqref{eventually}.}\label{eventuallyfigure}
\end{figure}
\begin{lem}\thlabel{firstaxisright}
 It holds $P$--a.s.\ that for all $x\in\N$ and all $m'>m\ge M$,
 \begin{equation}
  \IG^{\mathbf X}(xe_0)\neq2m+1\text{ and }\IG^{\mathbf X}(G_m^0e_0)\neq 2m'.
 \end{equation}
\end{lem}
\begin{proof}
 Take any $m\ge M$. We have
 \begin{equation}
  G_m^1\ge G_{m-1}^1+2\plannedwidth_m-\plannedwidth_{m-1}\ge2\plannedwidth_m-\plannedwidth_{m-1}\ge\plannedwidth_m.
 \end{equation}
 $\mathbf X(G_m^1,2m+1,0)=1$ induces a (part of a) street of superlevel $m$, with planned width $\beta_m$. Thus, this street of sublevel $1$ does not reach the zeroth axis. 

 Now take $m'>m$. We know that $G_{m'}^0\ge G_{m'-1}^0+\plannedwidth_{m'}\ge G_m^0+\plannedwidth_{m'}$. This shows that any vertical street of higher level does not reach $G_m^0e_0$.
\end{proof}
\begin{lem}\thlabel{evlevelright}
 \begin{equation}
  \SG_{G_m^je_j}=2m+j,\ m\ge M,\ j\in\{0,1\}.
 \end{equation}
\end{lem}
\begin{proof}
 We prove the case $j=0$. 

 Take $m\ge M$. As $G_m^0$ is a natural number such that $\mathbf X\big(G_m^0,2m,0\big)=1$, we have $W_{G_m^0}^0(\mathbf X)\ge2m$. $(G_n^0)_{n\ge M}$ is an increasing sequence. Thus, it holds that $V_{G_m^0}^0\ge2m$. This implies that 
 \begin{equation}
  \IG^{\mathbf X}(G_m^0e_0)\ge2m.\label{greaterstillpossible}
 \end{equation}
 The case ``$>$'' subdivides into the two 
 \begin{itemize}
  \item $\IG^{\mathbf X}(G_m^0e_0)=2m'+1$ for some $m'\ge m$,
  \item $\IG^{\mathbf X}(G_m^0e_0)=2m'$ for some $m'>m$,
 \end{itemize}
 which both are excluded by \thref{firstaxisright}. Hence, equality holds in \eqref{greaterstillpossible}. $G_m^j$ depends only on $\mathbf X(\cdot,\cdot,0)$. As all streets that are not fields w.r.t.\ $\IG$ remain untouched in the construction of the final streetgrid, the equality holds for $\SG_{G_m^0e_0}$ as well.

 Similar observations can be made for points of the form $G_m^1e_1$, $m\ge M$, using an adapted version of \thref{firstaxisright}.
\end{proof}
\begin{dfn}\thlabel{Srndindep}
 Set
 \begin{equation}
  B_m^j:=\Srnd^{\IG}(G_m^je_j)=\Srnd^{\SG}(G_m^je_j),\ m\ge M,\ j\in\{0,1\}.
 \end{equation}
\end{dfn}
\begin{cor}\thlabel{widthandlength}
 For all $m\ge M$,the width $b_j'(B_m^j)-b_j(B_m^j)+1$ of $B_m^j$ is larger than or equal to $\plannedwidth_m$, while the length of the intersection of $B_m^j$ and the first quadrant, $(\UR B_m^j)_i$, $i\neq j$, $i,j\in\{0,1\}$, satisfies
 \begin{equation}
  \ROO_{m+1}^{-\alpha}\ge\begin{cases}
                          (\UR B_m^0)_1\\
                          (\UR B_m^1)_0.
                         \end{cases}
 \end{equation}
\end{cor} 
\begin{proof}
 Both assertions follow from the same type of arguments as in the proof of the \thref{firstaxisright,evlevelright}; the second one makes also use of the upper bounds provided by \eqref{eventually}.
\end{proof}
\begin{cor}\thlabel{evstreetsjoin}
 It holds for all $m\ge M$ that 
 \begin{equation}
  \Srnd(\UR B_m^0+e_1)=B_m^1\text{ and }\Srnd(\UR B_m^1+e_0)=B_{m+1}^0.
 \end{equation}
\end{cor}
\begin{proof}
 We prove only the first assertion. 

 Let $m\ge M$. As we have seen in \thref{evlevelright}, $\levelof(B_m^0)=2m$.

 By definition, the street $B_m^0$ extends vertically until it is blocked by some horizontal higher--level--street. The superlevel of this street is greater as or equal to $m$, otherwise there would be no blocking. Any horizontal street $B_{m'}^1$ of level $m'>m$ does not interfere with $B_m^1$, because the $G_\cdot^1$ all keep their distance from each other (see \eqref{eventually}). So, the blocking indeed happens by $B_m^1$. 
\end{proof}
\subsection{Stationarity}\label{stationarity}
\begin{ntn}\thlabel{notation}
 Take $F:\{0,1\}^\calZ\rightarrow\Nnull^{\Z^2}$ a function. Note that the values $F(\configX):\Z^2\rightarrow\Nnull$ of this function are themselves functions $u\mapsto F(\configX)_u$. Let $I\subseteq\calZ$, $D\subseteq\Z^2$. For $\configx\in\{0,1\}^I$, $g\in\Nnull^D$, by the notation 
 \begin{equation}
  F(\configx)|_D=g, \label{F(x)|_D=g}
 \end{equation}
 we shall express that
 \begin{equation}
  \text{for all }\configX\in\{0,1\}^\calZ\text{ such that }\configX|_I=\configx\text{, it holds that }F(\configX)_u=g_u\text{ for all }u\in D.
 \end{equation}
 Here, $\configX|_I:I\rightarrow\{0,1\}$ denotes the usual restriction of the function $\configX:\calZ\rightarrow\{0,1\}$ on $I$. Notation \eqref{F(x)|_D=g} however is more restrictive than a mere restriction, because it is understood that on $D$, $F(\cdot)$ does not depend on the values at places in $\calZ\setminus I$.

 Also define $0|_I$ to be the constant mapping that assigns $0$ to any element in $I$.
\end{ntn}
\begin{lem}\thlabel{blockaround}
 For any box $B\ni0$, there is $P$--a.s.\ a block w.r.t.\ $\SG(\mathbf X)$ containing $B$:
 \begin{equation}
  P\Big(\bigcup_{\substack{D\subseteq\Z^2:\\B\subseteq D}}\{D\text{ is block w.r.t.\ }\SG(\mathbf X)\}\Big)=1.
 \end{equation}
\end{lem}
\begin{proof}
 Take $B$ a box containing the origin. For $j\in\{0,1\}$, define the random variables
 \begin{equation}
  d_j:=\max\{x\le b_j(B)|V_{x-1}^j>\levelof^{\SG}(B)\}\text{ and }d_j':=\min\{x\ge b_j'(B)|V_{x+1}^j>\levelof^{\SG}(B)\}.
 \end{equation}
 These are $P$--almost surely finite, and $B\subseteq D:=\{d_0,\dots,d_0'\}\times\{d_1,\dots,d_1'\}$. $D$ is a random set and a block w.r.t.\ $\IG(\cdot)$. The streets placed on $D$ by the iterative construction leading to $\SG$ are all of lower level than the minimum of the levels present in $\partial D$, so that the block-property is preserved.
\end{proof}
\begin{dfn}\thlabel{IJ}
 Let $D$ be a block w.r.t.\ $g\in\Nnull^{\closure D}$ such that $0\in D$. Note that this is more a condition on $g$ than on $D$. Define 
 \begin{gather}
  J_g:=\Big\{(y,2m+j,0)\big|\ m>\RL{\loben^g(D)},\ j\in\{0,1\},\ b_j(D)-1\le y\le b_j'(D)+\beta_m\Big\}\subseteq\calZ
 \shortintertext{and }
  I_g:=\Big\{(y,2m+j,u)\big|\ j\in\{0,1\},\ 0\le m\le\RL{\loben^g(D)},\ b_j(D)-1\le y\le b_j'(D)+\beta_m,\ u\in D\Big\}\subseteq\calZ.
 \end{gather} 
\end{dfn}
The dependence of $I_g$ and $J_g$ on $D$ is omitted because it can be considered implicit via $g$.

To explain the meaning of these two sets, we need to go into greater detail. 

Take a realization of $\mathbf X$. It is an element of $\{0,1\}^\calZ$, and leads to $\SG=\SG(\mathbf X)$. One can ask at which points in $\calZ$ the values of $\mathbf X$ may be changed without changing the outcome of $\SG$, or $\SG|_D$ for some fixed $D\subseteq\Z$. 

The other way around, given a certain realization $g\in\Nnull^{\closure D}$ of $\SG(\mathbf X)|_{\closure D}$, where $D\subseteq\Z^2$ is a box, one can ask about the set of realizations of $\mathbf X$ such that 
\begin{equation}
\SG(\mathbf X)|_{\closure D}=g.\label{SG=g}                                                                                                                                                                                                        \end{equation}
It turns out it is enough to look at the outcome of $\mathbf X$ on the two subsets $I_g$ and $J_g$ of $\calZ$ in order to decide whether the last equation is true or not. All points that are responsible in the sense of \thref{responsibility} are contained in $I_g$, and $\mathbf X(\cdot)$ being equal to $0$ at all points in $J_g$ stipulates the absence of big streets that are not supposed to be on $\closure D$.

\begin{lem}\thlabel{wecantamper}
 Under the hypotheses of \thref{IJ}, $I_g$ is finite, and $I_g\cap J_g=\emptyset$. Let $\configX\in\{0,1\}^\calZ$. If $\SG(\configX)|_{\closure D}=g$, then it holds that $\SG(\configX|_{I_g\cup J_g})|_{\closure D}=g$, in the notation of \eqref{F(x)|_D=g}. In other words, $\SG(\configX)|_{\closure D}=g$ does not depend on $\configX|_{\calZ\setminus(I_g\cup J_g)}$. Also, $\SG(\configX)|_{\closure D}=g$ implies $\configX|_{J_g}=0|_{J_g}$. Finally, $P(\mathbf X|_{J_g}\equiv0)>0$.
\end{lem}
\begin{proof}[Proof of \thref{wecantamper}]
 The first two assertions are obvious. $\SG(\configX)|_{\closure D}=g$ does hold or not no matter what the values of $\configX$ at the points $(y,2m+j,u)$ with
 \begin{itemize}
  \item $u\in \Z^2\setminus D$, $m\in\N$, $y\in\Z$, $j\in\{0,1\}$,
  \item $u\in D\setminus\{0\}$, $m\ge\RL{\lunten^g(D)}$, $y\in\Z$, $j\in\{0,1\}$,
  \item $u\in D$, $m<\RL{\lunten^g(D)}$, $y\le b_j(D)-1$ or $y\ge b_j'(D)+\beta_m$, $j\in\{0,1\}$,
  \item $u=0$, $m\ge\RL{\lunten^g(D)}$, $y\le b_j(D)-2$ or $y\ge b_j'(D)+\beta_m+1$, $j\in\{0,1\}$.
 \end{itemize}
 Let us look at the lines one at a time. 

 As $D$ is a block w.r.t.\ $g$, and $0\in D$, all four streets in $\partial D$ are already present in $\IG^{\mathbf x}$. $\IG^{\mathbf x}$ is only influenced by the values of $\mathbf x$ at points $(\cdot,\cdot,0)$. The streets w.r.t.\ $g$ in $D$ are either streets w.r.t.\ $\IG^{\mathbf x}$ or are influenced by the values of $\mathbf x$ at the upper right corners of fields w.r.t.\ $\IG^{\mathbf x}$ or the subsequent iteration steps in the construction. These fields are entirely contained in $D$, again because $D$ is a block w.r.t.\ $g$. This is why points $(\cdot,\cdot,u)$ with $u\not\in D$ have no influence.

 We just looked at the influence of points in the upper right corners of fields lying entirely in $D$. The streets they induce are all of lower level than the minimum level present in $\partial D$; higher levels are not even considered, and thus the values of $\configX$ at the points in the second line have no influence on the equation. 

 The values of points with lower level do have an influence, but only if the index of the Bernoulli--process is not too far from $D$; to be precice, neither left to the lower end in the $j$--th coordinate--direction of $D$, nor farther than one street--width to the right of the upper end of $D$.

 Similarily, the values at the origin do not have any influence if the index of the Bernoulli--process is too far from $\closure D$; this translates as slightly loosened boundaries in the last line. 

 All remaining points are contained in $I_g$ and $J_g$, which contain however some of the cases above as well. This proves that $\SG(\configX)|_{\closure D}=g$ does not depend on $\configX|_{\calZ\setminus(I_g\cup J_g)}$.

 The superlevels of the streets in $\closure D$ are \emph{per definitionem} bounded by $\RL{\loben^g(D)}$. If the equation $\SG(\configX)|_{\closure D}=g$ is to hold, it is trivially true that 
 \begin{equation}
  \text{there is no street w.r.t.\ $\SG(\configX)$ of higher superlevel than }\RL{\loben^g(D)}\text{ in }\closure D.\label{nohigherlevel}
 \end{equation}

 This condition \eqref{nohigherlevel} is equivalent to 
 \begin{equation}
  \configX\big(y,2m+j,0\big)=0 \text{ for all }b_j(D)-1\le y\le b_j'(D)+\beta_m,m>\RL{\loben^g(D)},j\in\{0,1\}.\label{Xzero}
 \end{equation}
 \eqref{Xzero} can be written as $\configX|_{J_g}\equiv0$, which can hence be seen as an equivalent to \eqref{nohigherlevel}.

 Finally,  we have, with some non-trivial, non-random constant $c$,
 \begin{equation}
  P(\mathbf X|_{J_g}\equiv0)=\smashoperator[l]{\prod_{m>\RL{\loben^g(B)}}}\smashoperator[r]{\prod_{j\in\{0,1\}}}(1-\lambda_m)^{b_j'(D)-b_j(D)+\beta_m+2}\ge c\smashoperator[l]{\prod_{m\ge1}}\smashoperator[r]{\prod_{j\in\{0,1\}}}(1-\lambda_m)^{\beta_m}=c\smashoperator{\prod_{m\ge1}}(1-\lambda_m)^{2\beta_m}.
 \end{equation}

 This value to be larger than zero is equivalent to 
 \[\sum_{m\ge1}\beta_m\ln(1-\lambda_m)>-\infty.\]

 But
 \[\beta_m\ln(1-\lambda_m)\sim \beta_m(-\lambda_m)=-\frac{m!^2}{(m+1)!^2}=-\frac1{(m+1)^2},\]
 and we can, by the finiteness of the sum, confirm positive $P_{\mathbf X|_{J_g}}$-measure for $0|_{J_g}$.
\end{proof}
\begin{dfn} We need to define some shift operators and related notations. Let $v\in\Z^2$ be the vector we want to shift by.

 For $D\subseteq\Z^2$, we write $D+v:=\{u+v\ |\ u\in D\}.$

 For $D\subseteq \Z^2$, $f\in\Nnull^D$, we define the shifted $\theta_vf\in\Nnull^{D+v}$ by 
 \begin{equation}
  (\theta_vf)_u:=f_{u-v}\text{ for all }u\in D+v.
 \end{equation}

 We also can shift elements $(x,l,u)\in\calZ$ by
 \begin{equation}
  \theta_v(x,l,u):=(x+v_{\SL{l}},l,u+v).
 \end{equation}

 A slightly different shift will sometimes be needed for elements of the form $(x,l,0)\in\calZ$, namely one that preserves the special role of the origin:
 \begin{equation}
  \vartheta_v(x,l,0):=(x+v_{\SL{l}},l,0).
 \end{equation}

 With these last two definitions at hand, we can shift the two $I_g$ and $J_g$ from \thref{IJ} in the standard way by
 \begin{align}
  \theta_vI_g&:=\{\theta_v(x,l,u)\ |\ (x,l,u)\in I_g\},\\
  \vartheta_vJ_g&:=\{\vartheta_v(x,l,0)\ |\ (x,l,0)\in J_g\}.
 \end{align}

 Finally, we shift whole configurations $\configx\in\{0,1\}^I$, $I\subseteq\calZ$ by defining
 \begin{equation}
  \theta_v\configx\big((x,l,u)\big):=\configx\big(\theta_{-v}(x,l,u)\big),\ (x,l,u)\in\theta_vI.
 \end{equation}
\end{dfn}
\begin{lem}\thlabel{shiftedcantamper}
 Let $D\ni0$ be a block w.r.t.\ $g\in\Nnull^{\closure D}$, $I_g$, $J_g$ from \thref{IJ}. Also take any $v$ such that $-v\in D$. Then, $I_{\theta_vg}=\theta_vI_g$ and $J_{\theta_vg}=\vartheta_vJ_g$, and for any $\configX\in\{0,1\}^\calZ$, $\SG(\configX)|_{\closure D+v}=\theta_vg$ implies $\SG(\configX|_{\theta_vI_g\cup\vartheta_vJ_g})|_{\closure D+v}=\theta_vg$.
\end{lem}
\begin{proof}
 The first two equalities are easy exercises; an important point is how $\vartheta_\cdot$ preserves the special role of the origin, but at a different position relative to the shifted box. 

 The second assertion then follows directly frome \thref{wecantamper}, which tells us that $\SG(\configX)|_{\closure D+v}=\theta_vg$ implies $\SG(\configX|_{I_{\theta_vg}\cup J_{\theta_vg}})|_{\closure D+v}=\theta_vg$.
\end{proof}
Figure \ref{Responsibility} gives an idea of how the responsibility changes when the point of reference (the origin) is changed. This sort of changing will be employed in \thref{bijmap} in order to create a configuration of $\{0,1\}^{I_g}$ that yields the same outcome of the final streetgrid's construction, only shifted. 
\begin{figure}[htb]
 \begin{center}
  \begin{tikzpicture}
 \filldraw[black!45!white](-2,-2)rectangle(-1.8,2);
  \filldraw[black!45!white](1.5,-2)rectangle(1.7,2);
  \filldraw[black!30!white](-3,1)rectangle(-2,1.1);
  \filldraw[black!30!white](-1.8,0.5)rectangle(1.5,0.6);
  \filldraw[black!30!white](-1.8,-1.5)rectangle(1.5,-1.4);
  \filldraw[black!30!white](1.7,1.5)rectangle(3,1.6);
  \filldraw[black!10!white](-3,-2)rectangle(-2,1);
  \filldraw[black!10!white](-3,2)rectangle(-2,1.1);
  \filldraw[black!10!white](-1.8,2)rectangle(1.5,0.6);
  \filldraw[black!10!white](-1.8,0.5)rectangle(1.5,-1.4);
  \filldraw[black!10!white](-1.8,-1.5)rectangle(1.5,-2);
  \filldraw[black!10!white](1.7,1.6)rectangle(3,2);
  \filldraw[black!10!white](1.7,1.5)rectangle(3,-2);
  \draw(-3,-2)rectangle(3,2);
  \draw[very thin,<->](-3.2,-1.3)--(3.2,-1.3);
  \draw[very thin,<->](2,-2.2)--(2,2.2);
  \draw(2,-1.3)node[anchor=north west]{$0$}; 
  \draw[thin](-0.05,0)--(0.05,0)(0,-0.05)--(0,0.05)(0,0)node[anchor=north west]{$-v$};
  \draw[->](-2,2)--+(210:0.5);
  \draw[->](-2,2)--(-2.3,1.1);
  \draw[->](-2,1)--+(210:0.5);
  \draw[->](1.5,2)--+(210:0.5);
  \draw[->](1.5,0.5)--+(210:0.5);
  \draw[->](1.5,-1.5)--+(210:0.5);
  \draw[->](3,2)--+(195:1);
  \draw[->](2,-1.3)--(2.5,1.5);
  \draw[->](2,-1.3)--+(30:0.5);
  \draw[->](2,-1.3)--(1.7,-0.7);
 \draw[->](1.5,2)--(1,0.6);
 \begin{scope}[xshift=6.8cm]
  \filldraw[black!45!white](-2,-2)rectangle(-1.8,2);
 \filldraw[black!45!white](1.5,-2)rectangle(1.7,2);
 \filldraw[black!30!white](-3,1)rectangle(-2,1.1);
 \filldraw[black!30!white](-1.8,0.5)rectangle(1.5,0.6);
 \filldraw[black!30!white](-1.8,-1.5)rectangle(1.5,-1.4);
 \filldraw[black!30!white](1.7,1.5)rectangle(3,1.6);
 \filldraw[black!10!white](-3,-2)rectangle(-2,1);
 \filldraw[black!10!white](-3,2)rectangle(-2,1.1);
 \filldraw[black!10!white](-1.8,2)rectangle(1.5,0.6);
 \filldraw[black!10!white](-1.8,0.5)rectangle(1.5,-1.4);
 \filldraw[black!10!white](-1.8,-1.5)rectangle(1.5,-2);
 \filldraw[black!10!white](1.7,1.6)rectangle(3,2);
 \filldraw[black!10!white](1.7,1.5)rectangle(3,-2);
 \draw(-3,-2)rectangle(3,2);
 \draw[very thin,<->](-3.2,0)--(3.2,0);
 \draw[very thin,<->](0,-2.2)--(0,2.2);
 \draw(0,0)node[anchor=south east]{$0$};
 \draw[thin](1.95,-1.3)--(2.05,-1.3)(2,-1.25)--(2,-1.35)(2,-1.3)node[anchor=north west]{$v$};
 \draw[->](-2,2)--+(210:0.5);
 \draw[->](-2,2)--(-2.3,1.1);
 \draw[->](-2,1)--+(210:0.5);
 \draw[->](1.5,2)--+(210:0.5);
 \draw[->](1.5,-1.5)--+(210:0.5);
 \draw[->](3,2)--+(195:1);
 \draw[->](3,2)--(2.5,1.6);
 \draw[->](3,1.5)--+(210:0.5);
 \draw[->](0,0)--+(250:0.5);
 \draw[->](0,0)--(0.5,0.5);
 \draw[->](0,0)--+(1.5,-0.8);
\end{scope}
\end{tikzpicture}
 \end{center}
\caption{Responsibility. If the base of some arrow is at $w$ and the tip points to some street of level $l$, then $w$ is responsible for the emplacement of the streets of level $l$ in $D$, where $D$ is the smallest block containing the street.}\label{Responsibility}
\end{figure}
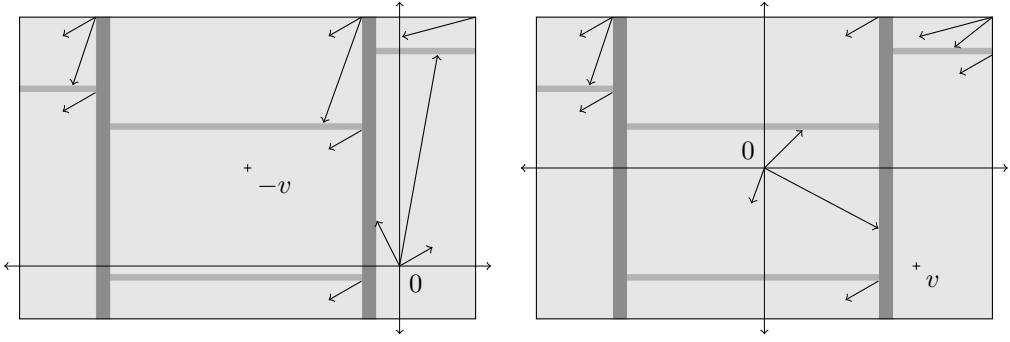
\begin{dfn}\thlabel{bijmap}
 Take the hypotheses of \thref{shiftedcantamper}. We define yet another operator on configurations on $\{0,1\}^{I_g}$, 
 \begin{equation}
 \configx\mapsto\ \Bijmapp\configx:
 \ \big\{\configy\in\{0,1\}^{I_g}\big|\SG(\configy,0|_{J_g})|_{\closure D}=g\big\}
 \rightarrow\big\{\configy\in\{0,1\}^{\theta_vI_g}\big|\SG(\configy,0|_{\vartheta_vJ_g})|_{\closure D+v}=\theta_vg\big\}.\label{bijmapequ}
 \end{equation}
 So, we need to define the object $\big(\Bijmapp\configx\big)(\cdot)$ for all $(y,l,w)\in I_{\theta_vg}$. We do this first for a special case of pairs $(l,w)$, and then for the rest. 

 Take any block $B$ w.r.t.\ $g$, and $\levelof^g(B)\le l<\lunten^g(B)$. Recall that $B+v$ is a block w.r.t.\ $\theta_vg$, and that $\levelof^g(B)=\levelof^{\theta_vg}(B+v)$ and $\lunten^g(B)=\lunten^{\theta_vg}(B+v)$. So, we can apply \thref{responsibility} and obtain $w\in D$ and $\widetilde w\in D+v$ such that
\begin{itemize}
 \item $\widetilde w$ responsible in $\theta_vg$ for the emplacement of streets of level $l$ in $B+v$, and
 \item $w$ responsible in $g$ for the emplacement of the streets of level $l$ in $B$,  
\end{itemize}
which are both the only points to satisfy these conditions. 

Write $m:=\RL{l}$ and $j:=\SL{l}$, and define, for $b_j(D+v)-1\le y\le b_j'(D+v)+\plannedwidth_m$,
\begin{equation}
 \big(\Bijmapp\configx\big)\big((y,l,\widetilde w)\big):=\configx\big((y-v_j,l,w)\big)\text{, and } \big(\Bijmapp\configx\big)\big((y,l,w+v)\big):=\configx\big((y-v_j,l,\widetilde w-v)\big).\label{firstdfn}
\end{equation}
For any other case that has not yet been covered, take $l<\lunten^g(D)$ and $\widetilde w\in D+v$ such that
 \begin{itemize}
  \item $\widetilde w$ is \emph{not} in $\theta_vg$ responsible for the emplacement of the streets of level $l$ in $\widetilde B$ for any block $\widetilde B$ w.r.t.\ $\theta_vg$, and 
  \item $\widetilde w-v$ is \emph{not} in $g$ responsible for the emplacement of the streets of level $l$ in $B$ for any block $B$ w.r.t.\ $g$;
 \end{itemize}
 then, with $m:=\RL{l}$, $j:=\SL{l}$, we define, for $b_j(D+v)-1\le y\le b_j'(D+v)+\beta_m$,
 \begin{equation}
  \big(\Bijmapp\configx\big)\big((y,l,\widetilde w)\big):=\configx\big((y-v_j,l,\widetilde w-v)\big).
 \end{equation}
\end{dfn}
This last definition shows that the operator $\Bijmapp$ is for most of the points really just the shift operator applied to the function $\configx$; only at the few points that are responsible, and at their counterparts in the shifted set, the special definition takes effect, and so to say, the responsibility is switched.

Note that $\Bijmapp$ depends strongly on $v$ and $g$, which gives again an implicit dependence on $D$. It also depends on our choice of $I_g$ and $J_g$. 
\begin{lem}\thlabel{bijmapfine}
 $\Bijmapp$ is well-defined, and takes indeed values in the specified codomain. It is bijective, and probability--preserving in the sense that  
 \begin{equation}
  P(\mathbf X|_{I_g}=\configx)=P(\mathbf X|_{\theta_vI_g}=\Bijmapp\configx)\text{ for all }\configx\in\{0,1\}^{I_g}. 
 \end{equation}
Also, the following equivalence holds: 
 \begin{equation}
  \SG(\configx,0|_{J_g})|_{\closure D}=g\Longleftrightarrow\SG(\Bijmapp\configx,0|_{\vartheta_vJ_g})|_{\closure D+v}=\theta_vg.
 \end{equation}
\end{lem}
\begin{proof}
 For the first part of the definition, we remark that if $\widetilde w=w+v$, the two definitions in \eqref{firstdfn} coincide: $w=\widetilde w-v$. So, the two do not contradict each other immediately. Also, given any $l$, $w$ and $\widetilde w$ are, respectively, responsible for $l$ only in $B$ and $B+v$. This was shown in \thref{responsibleforoneblock}. In the second part, the two bullets make sure that only cases not yet covered by the first part are defined. So, we indeed did not commit the error of multiply defining things. 

 The verification of $\Bijmapp\configx\in\{0,1\}^{\theta_vI_g}$ consists in checking that the domain of $\Bijmapp\configx$ is contained in $\theta_vI_g$. Indeed, the indices $y$ are chosen in the correct range. Also, $l<\lunten^g(B)<\lunten^g(D)$. Finally, $\widetilde w,w+v\in B+v\subseteq D+v$. The same applies to the second part of the definition. 

 To prove the bijectivity of $\Bijmapp\ =\Bijmapp(v,g)$, we consider the inverse function, which is $\Bijmapp(-v,\theta_vg)$. To check that this is true, remark that the two parts of the definition of $\Bijmapp$ can be inverted separately; the responsible points are just reversed, and the responsibilities switched back. The points which are not responsible being identical, the values there get shifted back as well. 

 For the preservation of probability, note that $\Bijmapp$ leaves the levels intact, and replicates the same number of zeros and ones, just at different places. Then, the stationarity of the Bernoulli--processes takes effect. 

 The last statement is a consequence of the concept of switching responsibilities described above. The operator moves the values of $\configx$ at any point responsible in $g$ for the emplacement of streets of level $l$ in $B$ to the point which is in $\theta_vg$ responsible for the emplacement of streets of level $l$ on $B+v$. If one translates the concept of responsibility into the construcion of the streetgrid, one sees that $\SG(\Bijmapp,0|_{J_{\theta_vg}})$ reconstitutes indeed the shifted $g$ on the shifted domain. 

 The opposite inclusion follows from the above considerations on bijectivity. 
\end{proof}
\begin{thm}
 $\SG(\mathbf X)$ is stationary.
\end{thm}
\begin{proof}
 We need to show the invariance of $\SG(\mathbf X)$'s finite-dimensional marginal distributions under the arbitrary shifts in $\Z^2$. Fortunately, we can restrict ourselves to distributions on boxes and shift-vectors inside these boxes: if we need a farther shift, we just take a bigger box.

 Let $B\ni0$ be a box, $v\in\Z^2$ such that $-v\in B$, and $g\in\Nnull^B$.

 In the following calculations, the first equality is due to \thref{blockaround}, the second one is true because the smallest (w.r.t.\ the semi-order established by the subset-relation) block around $B$ is unique. The fourth equality holds because the block property depends only on $\closure D$, and for the sixth one we apply \thref{wecantamper} for one inclusion, the other one following directly from \thref{notation}. \thref{wecantamper} also implies the disjointness of $I_{\widehat g}$ and $J_{\widehat g}$ leading to the independence used for the seventh equality. For the last equality, we apply \thref{bijmapfine}.
\begin{align}
  \MoveEqLeft P\big(\SG(\mathbf X)|_B=g\big)\\
     &=P\Big(\bigcup_{\substack{D\supseteq B\\\text{ box}}}\big\{\SG(\mathbf X)|_B=g,\ D\text{ block w.r.t.\ }\SG(\mathbf X)\big\}\Big)\\
  &=\sum_{\substack{D\supseteq B\\\text{ box}}}P\big(\SG(\mathbf X)|_B=g,\ D\text{ is the smallest block w.r.t.\ }\SG(\mathbf X)\text{ containing }B\big)\\
     &=\sum_{\substack{D\supseteq B\\\text{ box}}}\sum_{\substack{\widehat g\in\Nnull^{\closure D}:\\\widehat g|_B=g}}P\big(\SG(\mathbf X)|_{\closure D}=\widehat g,\ D\text{ is the smallest block w.r.t.\ }\SG(\mathbf X)\text{ containing }B\big)\\
  &=\sum_{\substack{D\supseteq B\\\text{ box}}}\sum_{\substack{\widehat g\in\Nnull^{\closure D}:\\\widehat g|_B=g}}P\big(\SG(\mathbf X)|_{\closure D}=\widehat g,\ D\text{ is the smallest block w.r.t.\ }\widehat g\text{ containing }B\big)\\
     &=\sum_{\substack{D\supseteq B\\\text{ box}}}\sum_{\substack{\widehat g\in\Nnull^{\closure D}:\\\widehat g|_B=g}}\indic_{D\text{ is the smallest block w.r.t.\ }\widehat g\text{ containing }B}P\big(\SG(\mathbf X)|_{\closure D}=\widehat g\big)\\
  &=\sum_{\substack{D\supseteq B\\\text{ box}}}\sum_{\substack{\widehat g\in\Nnull^{\closure D}:\\\widehat g|_B=g}}\indic_{D\text{ smallest block}}\sum_{\configx\in\{0,1\}^{I_{\widehat g}}}P\big(\SG(\configx,0|_{J_{\widehat g}})|_{\closure D}=\widehat g,\ \mathbf X|_{I_{\widehat g}}=\configx,\ \mathbf X|_{J_{\widehat g}}\equiv0\big)\\
     &=\sum_{\substack{D\supseteq B\\\text{ box}}}\sum_{\substack{\widehat g\in\Nnull^{\closure D}:\\\widehat g|_B=g}}\indic_{D\text{ smallest block}}\sum_{\configx\in\{0,1\}^{I_{\widehat g}}}P(\mathbf X|_{I_{\widehat g}}=\configx)P\big(\mathbf X|_{J_{\widehat g}}\equiv0\big)\indic_{\SG(\configx,0|_{J_{\widehat g}})|_{\closure D}=\widehat g}\\
  &=\smashoperator[l]{\sum_{\substack{D\supseteq B\\\text{ box}}}}\smashoperator[r]{\sum_{\substack{\widehat g\in\Nnull^{\closure D}:\\\widehat g|_B=g}}}\indic_{D\text{ smallest block}}\smashoperator{\sum_{\configx\in\{0,1\}^{I_{\widehat g}}}}P\big(\mathbf X|_{\theta_vI_{\widehat g}}=\Bijmapp\configx\big)P\big(\mathbf X|_{\vartheta_vJ_{\widehat g}}\equiv0\big)\indic_{\SG(\Bijmap\configx,0|_{\vartheta_vJ_{\widehat g}})|_{\closure D+v}=\theta_v\widehat g}
\intertext{We continue by applying the bijectivity of $\Bijmapp$, and reverting the steps which lead here, but with respect to the shifted sets.} 
  &=\sum_{\substack{D\supseteq B\\\text{ box}}}\smashoperator[r]{\sum_{\substack{\widehat g\in\Nnull^{\closure D}:\\\widehat g|_B=g}}}\indic_{D\text{ smallest block}}\smashoperator{\sum_{\configy\in\{0,1\}^{\theta_vI_{\widehat g}}}}P\big(\mathbf X|_{\theta_vI_{\widehat g}}=\configy\big)P\big(\mathbf X|_{\vartheta_vJ_{\widehat g}}\equiv0\big)\indic_{\SG(\configy,0|_{\vartheta_vJ_{\widehat g}})|_{\closure D+v}=\theta_v\widehat g}\\
     &=\sum_{\substack{D\supseteq B\\\text{ box}}}\smashoperator[r]{\sum_{\substack{\widehat g\in\Nnull^{\closure D}:\\\widehat g|_B=g}}}\indic_{D\text{ smallest block}}\smashoperator{\sum_{\configy\in\{0,1\}^{\theta_vI_{\widehat g}}}}P\big(\SG(\configy,0|_{\vartheta_vJ_{\widehat g}})|_{\closure D+v}=\theta_v\widehat g,\mathbf X|_{\theta_vI_{\widehat g}}=\configy,\mathbf X|_{\vartheta_vJ_{\widehat g}}\equiv0\big)\\
  &=\sum_{\substack{D\supseteq B\\\text{ box}}}\sum_{\substack{\widehat g\in\Nnull^{\closure D}:\\\widehat g|_B=g}}\indic_{D\text{ smallest block w.r.t.\ }\widehat g\text{ containing }B}P\big(\SG(\mathbf X)|_{\closure D+v}=\theta_v\widehat g\big)\\
  \intertext{At this point, we need to adjust the summation. We perform some trivial shift operations and change the indices of the sums:}
     &=\sum_{\substack{D\supseteq B\\\text{ box}}}\sum_{\substack{\widehat g\in\Nnull^{\closure D}:\\\widehat g|_B=g}}\indic_{D+v\text{ smallest block w.r.t.\ }\theta_v\widehat g\text{ containing }B+v}P\big(\SG(\mathbf X)|_{\closure D+v}=\theta_v\widehat g\big)\\
  &=\sum_{\substack{D\supseteq B\\\text{ box}}}\sum_{\substack{\widetilde g\in\Nnull^{\closure D+v}:\\\widetilde g|_{B+v}=\theta_vg}}\indic_{D+v\text{ smallest block w.r.t.\ }\widetilde g\text{ containing }B+v}P\big(\SG(\mathbf X)|_{\closure D+v}=\widetilde g\big)\\
     &=\sum_{\substack{D'\supseteq B+v\\\text{ box}}}\sum_{\substack{\widetilde g\in\Nnull^{\closure{D'}}:\\\widetilde g|_{B+v}=\theta_vg}}\indic_{D'\text{ smallest block w.r.t.\ }\widetilde g\text{ containing }B+v}P\big(\SG(\mathbf X)|_{\closure{D'}}=\widetilde g\big).\label{lastequalline}
 \end{align}  

 On the other hand, because $-v\in B$ implies $0\in B+v$, we can apply \thref{blockaround} to the following (with subsequent steps similar to the ones just performed): 
 \begin{align}
  \MoveEqLeft P\big(\SG(\mathbf X)|_{B+v}=\theta_vg\big)\\
  &=P\Big(\bigcup_{\substack{D'\supseteq B+v\\\text{ box}}}\big\{\SG(\mathbf X)|_{B+v}=\theta_vg,\ D'\text{ block w.r.t.\ }\SG(\mathbf X)\big\}\Big)\\
  &=\sum_{\substack{D'\supseteq B+v\\\text{ box}}}P\big(\SG(\mathbf X)|_{B+v}=\theta_vg,\ D'\text{ is the smallest block w.r.t.\ }\SG(\mathbf X)\text{ containing }B+v\big)\\
  &=\sum_{\substack{D'\supseteq B+v\\\text{ box}}}\smashoperator[r]{\sum_{\substack{\widetilde g\in\Nnull^{\closure{D'}}:\\\widetilde g|_{B+v}=\theta_vg}}}P\big(\SG(\mathbf X)|_{\closure{D'}}=\widetilde g,\ D'\text{ is the smallest block w.r.t.\ }\SG(\mathbf X)\text{ containing }B+v\big)\\
  &=\sum_{\substack{D'\supseteq B+v\\\text{ box}}}\smashoperator[r]{\sum_{\substack{\widetilde g\in\Nnull^{\closure{D'}}:\\\widetilde g|_{B+v}=\theta_vg}}}P\big(\SG(\mathbf X)|_{\closure{D'}}=\widetilde g,\ D'\text{ is the smallest block w.r.t.\ }\widetilde g\text{ containing }B+v\big),
\end{align}
which is equal to \eqref{lastequalline}.
\end{proof}
\subsection{Mixing and ergodic properties}\label{mixing}
\begin{dfn}
 We say a family $(F_u)_{u\in\Z^2}$ of discrete random variables on the probability space $(\Omega,\filt F,\P)$ is \emph{mixing w.r.t.\ $\P$} if for any $v\in\Z^2\setminus0$, any finite box $B\subseteq\Z^2$, and any realizations $f_1,f_2:B\rightarrow\R$, it holds that
 \begin{equation}
  \big|\P(F|_B=f_1,F|_{B+nv}=\theta_{nv}f_2)-\P(F|_B=f_1)\P(F|_{B+nv}=\theta_{nv}f_2)\big|\xrightarrow[n\to\infty]{}0.
 \end{equation}
\end{dfn}
\begin{thm}\thlabel{SGmixing}
 The streetgrid $\SG(\mathbf X)$ is mixing w.r.t.\ $P$.
\end{thm}
\begin{proof}
 Take $B$ a box. Because of the stationarity of the process $\SG(\mathbf X)$, we can suppose $0\in B$ without loss of generality. Let $g\in\Nnull^B$. As for the shift, take $v\in\Z^2$ such that $v_0>0$. The case $v_1>0$ can be proven analogously. 

 Define the \emph{cutting--event}
\begin{equation}
 C_n:=\big\{\exists m>\RL{\levelof^{\SG}(B)}\vee\RL{\levelof^{\SG}(B+nv)},\ \exists\ b_0'(B)+\plannedwidth_m\le x\le b_0(B+nv)-1:\mathbf X(x,2m,0)=1\big\}.
\end{equation}
The meaning of this event is that between $B$ and $B+nv$, there is a vertical street of higher level than any of the streets in $g$ and $h$. 

The event $C_n$ satisfies, for $n\in\N$ large enough,
\begin{align}
 \MoveEqLeft P(C_n)=P\Big(\bigcup_{m>\RL{\levelof^{\SG}(B)}\vee\RL{\levelof^{\SG}(B+nv)}}\quad\smashoperator[r]{\bigcup_{b_0'(B)+\plannedwidth_m\le x\le b_0(B+nv)-1}}\ \big\{\mathbf X\big(x,2m,0\big)=1\big\}\Big)\\
 &=\sum_{\widehat m\in\Nnull}P\Big(\bigcup_{m>\widehat m}\quad\smashoperator[r]{\bigcup_{b_0'(B)+\plannedwidth_m\le x\le b_0(B+nv)-1}}\ \big\{\mathbf X\big(x,2m,0\big)=1\big\}\Big)P\Big(\RL{\levelof^{\SG}(B)}\vee\RL{\levelof^{\SG}(B+nv)}=\widehat m\Big)\\
 &\ge\sum_{\widehat m\in\N}P\Big(\smashoperator[r]{\bigcup_{b_0'(B)+\plannedwidth_{\widehat m}\le x\le b_0(B+nv)-1}}\ \big\{\mathbf X\big(x,2\widehat m,0\big)=1\big\}\Big)P\Big(\RL{\levelof^{\SG}(B)}\vee\RL{\levelof^{\SG}(B+nv)}=\widehat m-1\Big)\\
 &=\sum_{\widehat m\in\N}\Big(1-(1-\ROO_{\widehat m})^{nv_0+b_0(B)-b_0'(B)-\plannedwidth_{\widehat m}}\Big)P\Big(\RL{\levelof^{\SG}(B)}\vee\RL{\levelof^{\SG}(B+nv)}=\widehat m-1\Big)\xrightarrow[n\to\infty]{}1.
\end{align}
$C_n$ also has the property to render independent events happening on $B$ and $B+nv$: it implies that the smallest block around $B$ w.r.t.\ $\SG(\mathbf X)$ and the smallest block around $B+nv$ w.r.t.\ $\SG(\mathbf X)$ are disjoint, which means that different points are responsible for the two. We hence have, with the events
\begin{equation}
 G:=\{\SG(\mathbf X)|_B=g\}\text{ and }H_v:=\{\SG(\mathbf X)|_{B+v}=\theta_vh\},
\end{equation}
and if we denote by $C_n^c$ the complement of $C_n$,
\begin{align}
 \MoveEqLeft P(\SG(\mathbf X)|_B=g,\SG(\mathbf X)|_{B+nv}=\theta_{nv}h)\\
 &=P(G\cap H_{nv})\\
 &=P(G\cap H_{nv}\cap C_n^c)+P(G\cap H_{nv}\cap C_n)\\
 &=P(G\cap H_{nv}\cap C_n^c)+P(G\cap H_{nv}|C_n)P(C_n)\\
 &=P(G\cap H_{nv}\cap C_n^c)+P(G|C_n)P(H_{nv}|C_n)P(C_n)\\
 &=P(G\cap H_{nv}\cap C_n^c)+P(G\cap C_n)P(H_{nv}\cap C_n)/P(C_n)\\
 &\xrightarrow[n\to\infty]{}P(G)P(H_0).
\end{align}
\end{proof}
\begin{cor}
 The streetgrid $\SG(\mathbf X)$ is totally ergodic.
\end{cor}
\subsection{Consequences}
\begin{cor}
 The environment $\omega$ is stationary.
\end{cor}
\begin{proof}
 This is true because in order to determine every point $\omega(u)$, $u\in\Z^2$, the same function is applied to the $\SG$--values around $u$ in a  local and stationary manner, and because $\SG(\mathbf X)$ is stationary. 
\end{proof}
\begin{cor}
 The environment $\omega$ is mixing.
\end{cor}
\begin{proof}
 To prove this, we would like to carry over the arguments from the proof of \thref{SGmixing}. However there is an issue about $\omega$ (as a function of $\SG$) not being completely localized in the sense that in order to determine $\omega$ on a box $B$, one needs to know the width of the streets present in $B$. Recall that only if a street has its full planned width the biased transition probabilities are placed on it; else, the transition probabilities of a simple random walk are used. 
 
 Fortunately, it is possible to determine what $\omega$ looks like on $B$ by knowing $\SG$ on a box 
 \begin{equation}
  \overline B^{\plannedwidth_{\levelof^{\SG}(B)}}:=\{b_0(B)-\plannedwidth_{\levelof^{\SG}(B)},\dots,b_1(B)-\plannedwidth_{\levelof^{\SG}(B)}\}\times\{b_0'(B)-\plannedwidth_{\levelof^{\SG}(B)},\dots,b_1'(B)+\plannedwidth_{\levelof^{\SG}(B)}\}; 
 \end{equation}
 one migth want to think of this box as a thicker closure, with thickness $\plannedwidth_{\levelof^{\SG}(B)}$. In other words, $\omega|_B$ is $\SG(B)|_{\overline B^{\plannedwidth_{\levelof^{\SG}(B)}}}$--measurable. 

 We will use this fact in the following calculations. Take $\widetilde g,\widetilde h\in(S^2)^B$.
\begin{align}
 \MoveEqLeft P\big(\omega|_B=\widetilde g,\ \omega|_{B+nv}=\theta_{nv}\widetilde h\big)\\
 &=\sum_{k,l\in\N}P\big(\omega|_B=\widetilde g,\ \omega|_{B+nv}=\theta_{nv}\widetilde h,\ \levelof^{\SG}(B)=k,\ \levelof^{\SG}(B+nv)=l\big)\\
 &=\sum_{k,l\in\N}\smashoperator[r]{\sum_{g\in\N^{\overline B^{\plannedwidth_k}},\ h\in\N^{\overline B^{\plannedwidth_l}}}}P\big(\omega|_B=\widetilde g,\ \omega|_{B+nv}=\theta_{nv}\widetilde h,\ \levelof^{\SG}(B)=k,\ \levelof^{\SG}(B+nv)=l,\\
   &\qquad\qquad\qquad\qquad\qquad\qquad\qquad\qquad\qquad\qquad\qquad \SG|_{\overline B^{\plannedwidth_k}}=g,\ \SG|_{\overline B^{\plannedwidth_l}+nv}=\theta_{nv}h\big)\\
 &=\sum_{k,l\in\N}\sum_{\substack{g\in\N^{\overline B^{\plannedwidth_k}}\\h\in\N^{\overline B^{\plannedwidth_l}}}}P\Bigg(\splitfrac{\omega|_B=\widetilde g,}{\omega|_{B+nv}=\theta_{nv}\widetilde h,}\splitfrac{\levelof^{\SG}(B)=k,}{\levelof^{\SG}(B+nv)=l}\bigg|\splitfrac{\SG|_{\overline B^{\plannedwidth_k}}=g,}{\SG|_{\overline B^{\plannedwidth_l}+nv}=\theta_{nv}h}\Bigg)\\
   &\qquad\qquad\qquad\qquad\qquad\qquad\qquad\qquad\qquad\qquad\qquad P\big(\SG|_{\overline B^{\plannedwidth_k}}=g,\ \SG|_{\overline B^{\plannedwidth_l}+nv}=\theta_{nv}h\big)\\
 &=\sum_{k,l\in\N}\smashoperator[r]{\sum_{g\in\N^{\overline B^{\plannedwidth_k}},\ h\in\N^{\overline B^{\plannedwidth_l}}}}\indic_{\omega(g)|_B=\widetilde g,\ \omega(\theta_{nv}h)|_{B+nv}=\theta_{nv}\widetilde h,\ \levelof^{g}(B)=k,\ \levelof^{\theta_{nv}h}(B+nv)=l}\\
   &\qquad\qquad\qquad\qquad\qquad\qquad\qquad\qquad\qquad\qquad\qquad P\big(\SG|_{\overline B^{\plannedwidth_k}}=g,\ \SG|_{\overline B^{\plannedwidth_l}+nv}=\theta_{nv}h\big)\\ 
 &\xrightarrow[n\to\infty]{}\sum_{k,l\in\N}\sum_{\substack{g\in\N^{\overline B^{\plannedwidth_k}}\\h\in\N^{\overline B^{\plannedwidth_l}}}}\indic_{\omega(g)|_B=\widetilde g,\ \levelof^{g}(B)=k}\indic_{\omega(h)|_{B}=\widetilde h,\ \levelof^{h}(B)=l}P\big(\SG|_{\overline B^{\plannedwidth_k}}=g\big)P\big(\SG|_{\overline B^{\plannedwidth_l}}=h\big)\\ 
 &=P\big(\omega|_B=\widetilde g\big)P\big(\omega|_B=\widetilde h\big)
\end{align}
\end{proof}

\section{Properties of the random walk}\label{largeindeed}
\subsection{The main theorem and the idea of its proof}
 \begin{thm}\thlabel{Xtoinfty}
  \begin{equation}
   PP_0^\omega(X_t\cdot\vec1\xrightarrow[t\to\infty]{}\infty)>0, 
  \end{equation}
  where $\omega$ is the environment from \thref{environment} with its corresponding probability measure $P$, and $(X_t,P_0^\omega)$ the random walk from \eqref{Pomega(X)}.
 \end{thm}
 A similar assertion holds with $\vec1$ replaced by $-\vec1$. The two together imply \thref{mainth}.

 Recall the heuristical description at the beginning of Subsection \ref{large}. The idea of the proof of the \thnameref{Xtoinfty} is that the random walk $X_t$ has positive probability to follow the streets in the initial grid $\IG$, at least from some starting point onwards. The starting point has positive probability to be reached directly from the origin. From there the random walk proceeds exactly like described, except for the ``going straight'' part: as it is a \emph{random} walk, we have to take care of some fluctuations; but this is possible thanks to the streets growing nicely, see \thref{widthandlength}.

 A complete proof of \thref{Xtoinfty} will be given later. We start with a few technical 
\subsection{Definitions and Lemmata}  
\begin{dfn} \thlabel{hitting}
 We define the hitting time of the random walk $(X_t)_t$ of the set $B\subseteq\Z^2$ as
  \begin{equation}
   \tau_B:=\inf\{t\ge0|X_t\in B\},
  \end{equation}
  and the hitting time of the set $B'\subseteq\Z^2$ after hitting $B$ as 
  \begin{equation}
   \tau_{B,B'}=\tau(B,B'):=\inf\{t\ge\tau_B|X_t\in B'\}.
  \end{equation}
  $\tau_B$ and $\tau_{B,B'}$ are of course stopping times w.r.t.\ $\filt G_t:=\sigma(X_s,s\le t)$ the natural filtration. 
 \end{dfn}
 \begin{dfn}\thlabel{manysets}
  We define sequences of sets, some of which depend on the parameter $n\in\N$: 
  \begin{align}
   \mathcal B_m^{\mathtt a}(n)&:=\{-\frac{\plannedwidth_m}{16}+1,\dots,\frac{\plannedwidth_m}{16}\}\times\{-\frac{\plannedwidth_{m-1}}{16},\dots,n\},\\
   \mathcal S_m^{\mathtt a}&:=\Btwn\big(0,\frac{\plannedwidth_m}2e_1\big),\\
   \mathcal E_m^{\mathtt a}(n)&:=\{u\in\partial\mathcal B_m^{\mathtt a}(n)|u_1\le n\},\ m\ge5.
  \end{align}
  ``$\mathcal S$'' and ``$\mathcal E$'' stand for ``Start'' and ``Escape''. Furthermore, define the ``Target''-set
  \begin{equation}
   \mathcal T_m^{\mathtt a}(n):=\partial\mathcal B_m^{\mathtt a}(n)\setminus\mathcal E_m^{\mathtt a}(n)=\Btwn\big((-\frac{\plannedwidth_m}{16}+1,n+1),(\frac{\plannedwidth_m}{16},n+1)\big),\ m\ge5.
  \end{equation}
  The reason for the restriction to $m\ge5$ is the same as in \eqref{dfnM}.
 \end{dfn}
 \begin{lem}\thlabel{escapen1}
  Take some sequence $(n_m)_{m\ge5}$ such that $\frac{\plannedwidth_m}2\le n_m\le\plannedwidth_{m+2}^\alpha$, $m\ge5$. Also take a sequence of starting points $v_m\in\mathcal S_m^{\mathtt a}$, $m\ge5$. We consider the (non-random) environment defined by setting
 \begin{equation}
  \varpi^{\mathtt a}(u):=\begin{cases}
               \omega_{\nearrow}^0&\text{if }u_0\le0,\\
               \qquad\omega_{\nwarrow}^0&\text{if }u_0>0
             \end{cases}
 \end{equation}
 for all $u\in\Z^2$. It engenders the random walk $(X_t)_{t\ge0}$ in the environment $\varpi^{\mathtt a}$, starting in $v_m$, given by the measure $P_{v_m}^{\varpi^{\mathtt a}}$. It now holds that $P_{v_m}^{\varpi^{\mathtt a}}(\tau_{\mathcal E_m^{\mathtt a}(n_m)}<\tau_{\mathcal T_m^{\mathtt a}(n_m)})$ is summable in $m$, where $\tau_\cdot$ is from \thref{hitting}.
 \end{lem}
 A picture of the sets from \thref{manysets} and the environment of \thref{escapen1} can be found in Figure \ref{escapeevents1}.
 \begin{figure}[tbh]
\begin{center}
\begin{tikzpicture}[scale=0.6]
\def\escapedecor{[red,decorate,decoration={zigzag,segment length=1.5mm,amplitude=0.5mm}]}
\def\uebergang[#1]#2{\begin{tikzpicture}[rotate=#1,xscale=#2, scale=0.5]\draw[<->](-0.19,0)--(0.75,0);\draw[<->](0,-0.19)--(0,0.75); \filldraw[fill=black](0,0)circle(0.05cm);\end{tikzpicture}}
\begin{scope}
    \def\oneeigthm{3}
    \def\oneeigthmmo{1.2}
    \def\tothealpha{14}
    \draw[<->](0,-\oneeigthmmo)+(0,-0.5)--(0,\tothealpha+0.5);
    \draw[<->](-\oneeigthm-0.5,0)--(\oneeigthm+0.5,0);
    \draw (0,0) node[anchor=south west]{$0$};
    \draw[ultra thick](0,0)--(0,4*\oneeigthm-2*\oneeigthmmo)+(0,-3);
    \draw[thick](0,-\oneeigthmmo)node[anchor=south east]{\strut$-\frac{\plannedwidth_{m-1}}{16}$}+(0.2,0) -- +(-0.2,0);
    \draw[thick](0,4*\oneeigthm-2*\oneeigthmmo)node[anchor=west]{$\frac{\plannedwidth_m}2$}+(0.2,0) -- +(-0.2,0);
    \draw[thick](0,\tothealpha) node[anchor=north east]{\strut$n$}+(0.2,0) -- +(-0.2,0);
    \draw[thick](-\oneeigthm-0.1,0) node[anchor=south west]{\strut$-\frac{\plannedwidth_m}{16}+1$}(-\oneeigthm,0)+(0,0.2) -- +(0,-0.2);
    \draw[thick](\oneeigthm+0.1,0) node[anchor=south east]{\strut$\frac{\plannedwidth_m}{16}$}(\oneeigthm,0)+(0,0.2) -- +(0,-0.2);
    \draw[very thin](-\oneeigthm,-\oneeigthmmo) rectangle (\oneeigthm,\tothealpha);
    \draw\escapedecor(-\oneeigthm,-\oneeigthmmo-0.1)--(\oneeigthm,-\oneeigthmmo-0.1);
    \draw\escapedecor(\oneeigthm+0.1,-\oneeigthmmo)--(\oneeigthm+0.1,\tothealpha);
    \draw\escapedecor(-\oneeigthm-0.1,-\oneeigthmmo)--(-\oneeigthm-0.1,\tothealpha);
    \draw[green](-\oneeigthm,\tothealpha+0.1)--(\oneeigthm,\tothealpha+0.1);
    \draw(-0.1,6.5)node[anchor=east]{$\mathcal S_m^{\mathtt a}$};
    \draw(1,\tothealpha)node[anchor=south]{$\mathcal T_m^{\mathtt a}(n)$};
    \draw(1,-\oneeigthmmo-0.1)node[anchor=north]{$\mathcal E_m^{\mathtt a}(n)$};
    \draw(1.5,4)node{\uebergang[90]{1}};
    \draw(-1.5,4)node{\uebergang[0]{1}};
\end{scope}
\begin{scope}[shift={(11,10.5)},scale=0.9]
      \def\onen{3}
      \fill[fill=green!20!white](-\onen,0)--(\onen,0) -- (2*\onen,\onen) -- (0,\onen) -- cycle;
      \draw[<->](0,-\onen.5)--(0,\onen.5);
      \draw[<->](-2*\onen-0.5,0)--(3*\onen+0.5,0);
      \draw (0,0) node[anchor=north east]{\strut$0$};
      \draw[thick](0,-\onen) node[anchor=south east]{\strut$-\frac{\plannedwidth_m}{16}$}+(0.2,0) -- +(-0.2,0);
      \draw[thick](0,\onen) node[anchor=north east]{\strut$\frac{\plannedwidth_m}{16}-1$}+(0.2,0) -- +(-0.2,0);
      \draw[thick](-2*\onen,0) node[anchor=south west]{\strut$-\frac{\plannedwidth_m}8+1$}+(0,0.2) -- +(0,-0.2);
      \draw[thick](-\onen,0)node[anchor=north]{$-\frac{\plannedwidth_m}{16}+1$}+(0,0.2) -- +(0,-0.2);
      \draw[thick](\onen,0)node[anchor=north]{$\frac{\plannedwidth_m}{16}$}+(0,0.2) -- +(0,-0.2);
      \draw[thick](3*\onen,0) node[anchor=south east]{\strut$\frac{3\plannedwidth_m}{16}$}+(0,0.2) -- +(0,-0.2);
      \draw[very thin](-2*\onen,-\onen) rectangle (3*\onen,\onen);
      \draw[ultra thick](-\onen,0)--(\onen,0)+(-2,0)node[anchor=north west]{};
      \draw[green](-\onen,\onen.1)--(3*\onen,\onen.1);
      \draw\escapedecor(-2*\onen,\onen.1)--(-\onen,\onen.1);
      \draw\escapedecor(-2*\onen-0.1,-\onen)--(-2*\onen-0.1,\onen);
      \draw\escapedecor(-2*\onen,-\onen-0.1)--(3*\onen,-\onen.1);
      \draw\escapedecor(3*\onen+0.1,-\onen)--(3*\onen+0.1,\onen); 
      \draw(\onen/2,-0.1)node[anchor=north]{$\mathcal S_m^{\mathtt b}$};
      \draw(\onen/2,\onen+0.1)node[anchor=south]{$\mathcal T_m^{\mathtt b}$};
      \draw(\onen/2,-\onen-0.1)node[anchor=north]{$\mathcal E_m^{\mathtt b}$};
      \draw(2.5,2)node[anchor=north]{\uebergang[270]{-1}};
      \draw(-2.5,2)node[anchor=north]{\uebergang[270]{-1}};
      \draw(2.5,-2.5)node[anchor=south]{\uebergang[90]{1}};
      \draw(-2.5,-2.5)node[anchor=south]{\uebergang[0]{1}};
\end{scope}
\begin{scope}[shift={(6,-1)},scale=0.85]
       \def\ones{1.5}
       \def\onee{3}
       \def\onef{6}
       \def\oneh{12}
       \def\one{24}
       \fill[fill=green!20!white](3*\ones,\ones)--(7*\ones,\ones) -- (10*\ones,4*\ones) -- (6*\ones,4*\ones) -- cycle;
       \draw[->](-0.2,0)--(11*\ones+0.5,0);
       \draw[->](0,-.2)--(0,4*\ones+0.5);
       \draw (0,0) node[anchor=north east]{\strut$0$};
       \draw[very thin](0,0) rectangle (11*\ones,4*\ones);
       \draw[ultra thick](3*\ones,\ones)--(7*\ones,\ones)node[anchor=south east]{};
       \draw\escapedecor(0,-0.1)--(11*\ones,-0.1);
       \draw\escapedecor(-0.1,0)--(-0.1,\onef);
       \draw\escapedecor(0,4*\ones+0.1)--(3*\ones,4*\ones+0.1);
       \draw\escapedecor(11*\ones+0.1,0)--(11*\ones+0.1,4*\ones);
       \draw[green](3*\ones,4*\ones+0.1)--(11*\ones,4*\ones+0.1);
       \draw[thick](11*\ones,-0.1) node[anchor=north]{\strut$\frac{11\plannedwidth_m}{16}$}(11*\ones,0)+(0,0.2) -- +(0,-0.2);
       \draw[thick](3*\ones,-0.1) node[anchor=north]{\strut$\frac{3\plannedwidth_m}{16}$}(3*\ones,0)+(0,0.2) -- +(0,-0.2);
       \draw[thick](7*\ones,-0.1) node[anchor=north]{\strut$\frac{7\plannedwidth_m}{16}-1$}(7*\ones,0)+(0,0.2) -- +(0,-0.2);
       \draw[thick](0,\onef) node[anchor=north west]{$\frac{\plannedwidth_m}4-2$}(0,\onef)+(0.2,0) -- +(-0.2,0);
       \draw[thick](0.1,\ones) node[anchor=west]{$\frac{\plannedwidth_m}{16}$}(0,\ones)+(0.2,0) -- +(-0.2,0);
       \draw(5*\ones,\ones-0.1)node[anchor=north]{$\mathcal S_m^{\mathtt c}$};
       \draw(5*\ones,\onef+0.1)node[anchor=south]{$\mathcal T_m^{\mathtt c}$};
       \draw(5*\ones,-0.1)node[anchor=north]{$\mathcal E_m^{\mathtt c}$};
       \draw(3,3)node[]{\uebergang[270]{-1}};
\end{scope}
\end{tikzpicture}
\caption{Escape and target sets used in \thref{escapen1,escapen2,escapen3}, together with their corresponding environment. Nothing is to scale.}\label{escapeevents1}
\end{center}
\end{figure}
 \begin{proof}
  We split the movement of $X_t$ into its two coordinates $X_t=(X_{t,0},X_{t,1})$. $X_{t,1}$ is stochastically minorated by a random walk on $\Z$ with uniform drift to the right (and possibility to sometimes stand still). The probability of this random walk to hit some negative $-a$ before wandering off towards infinity decays exponentially in $a$. 
 
  Also, the time to reach some positive $b$ grows linearily in $b$, in the sense that there is a positive, non-random constant $c_1$ such that the probability of not reaching $b$ up to time $c_1b$ decays exponentially fast in $b$. 

  As the probabilities set in $\varpi^{\mathtt a}$ to go left or right are uniformly bounded away from 1, the random walk $X_\cdot$ will spend a nontrivial fraction of its time going left and right. This means that there is some positive, non-random constant $c_2<1$ such that the probability that the number of times $X_\cdot$ goes left or right up to time $t$ is greater than $c_2t$ decays exponentially in $t$. 

  $|X_{t,0}|$ is stochastically dominated by a random walk reflected at $0$ with negative drift. Each excursion from $0$ of such a reflected random walk is stochastically dominated by a geometric random variable, and the excursions are independent; recall that the probability of a geometric random variable to be larger than $a$ decays exponentially in $a$.

  The number of excursions of $|X_{t,0}|$ up to some time can be estimated very crudely by the number of steps to the left or right up to that time. 
 
  If we put the pieces together, we find that the probability of escape to the left or right is for large $m$ bounded by the probability of at least one out of $c_2c_1(\plannedwidth_{m+2})^\alpha\ge c_2c_1n_m$ independent geometric random variables being larger than $\frac{\plannedwidth_m}{16}$, which can be verified to be still exponentially small in $m$. 

  As we did not care to keep track of exact rates, we settle for a much weaker statement of summability. 
 \end{proof}
 \begin{dfn}
  We need many more similar objects as the ones in \thref{manysets}:
  \begin{align}
   \mathcal B_m^{\mathtt b}&:=\{-\frac{\plannedwidth_m}8+1,\dots,\frac{3\plannedwidth_m}{16}\}\times\{-\frac{\plannedwidth_m}{16},\dots,\frac{\plannedwidth_m}{16}-1\},\\
   \mathcal S_m^{\mathtt b}&:=\Btwn\big((-\frac{\plannedwidth_m}{16}+1)e_0,\frac{\plannedwidth_m}{16}e_0\big),\\
   \mathcal E_m^{\mathtt b}&:=\{u\in\partial\mathcal B_m^{\mathtt b}|u_0\le-\frac{\plannedwidth_m}{16}\text{ or }u_1\le\frac{\plannedwidth_m}{16}-1\},\\
   \\
   \mathcal B_m^{\mathtt c}&:=\{0,\dots,\frac{11\plannedwidth_m}{16}\}\times\{0,\dots,\frac{\plannedwidth_m}4-2\},\\
   \mathcal S_m^{\mathtt c}&:=\Btwn\big((\frac{3\plannedwidth_m}{16},\frac{\plannedwidth_m}{16}),(\frac{7\plannedwidth_m}{16}-1,\frac{\plannedwidth_m}{16})\big),\\
   \mathcal E_m^{\mathtt c}&:=\{u\in\partial\mathcal B_m^{\mathtt c}|u_0\le\frac{3\plannedwidth_m}{16}-1\text{ or }u_1\le\frac{\plannedwidth_m}4-2\},\\
   \\
   \mathcal B_m^{\mathtt A}(n)&:=\{-\frac{\plannedwidth_m}{16},\dots,n\}\times\{-\frac{\plannedwidth_m}{16}+1,\dots,\frac{\plannedwidth_m}{16}\},\\
   \mathcal S_m^{\mathtt A}&:=\Btwn\big(0,\frac{\plannedwidth_m}2e_0\big),\\
   \mathcal E_m^{\mathtt A}(n)&:=\{u\in\partial\mathcal B_m^{\mathtt A}(n)|u_0\le n\},\\
   \\
   \mathcal B_m^{\mathtt B}&:=\{-\frac{\plannedwidth_m}{16},\dots,\frac{\plannedwidth_m}{16}-1\}\times\{-\frac{\plannedwidth_m}8+1,\dots,\frac{3\plannedwidth_m}{16}\},\\
   \mathcal S_m^{\mathtt B}&:=\Btwn\big((-\frac{\plannedwidth_m}{16}+1)e_1,\frac{\plannedwidth_m}{16}e_1\big),\\
   \mathcal E_m^{\mathtt B}&:=\{u\in\partial\mathcal B_m^{\mathtt B}|u_0\le\frac{\plannedwidth_m}{16}-1\text{ or }u_1\le-\frac{\plannedwidth_m}{16}\},\\
   \\
   \mathcal B_m^{\mathtt C}&:=\{0,\dots,\frac{\plannedwidth_{m+1}}4-2\}\times\{0,\dots,\frac{\plannedwidth_{m+1}}2+\frac{3\plannedwidth_m}{16}\},\\
   \mathcal S_m^{\mathtt C}&:=\Btwn\big((\frac{\plannedwidth_m}{16},\frac{3\plannedwidth_m}{16}),(\frac{\plannedwidth_m}{16},\frac{7\plannedwidth_m}{16}-1)\big),\\
   \mathcal E_m^{\mathtt C}&:=\{u\in\partial\mathcal B_m^{\mathtt C}|u_0\le\frac{\plannedwidth_{m+1}}4-2\text{ or }u_1\le\frac{3\plannedwidth_m}{16}-1\},\ m\ge5.
  \end{align}
  The target sets are
  \begin{gather}
   \mathcal T_m^\dagger:=\partial\mathcal B_m^\dagger\setminus\mathcal E_m^\dagger,\ \dagger\in\{\text{`` $\mathtt b$'',`` $\mathtt c$'',`` $\mathtt B$'',`` $\mathtt C$''}\},\\
   \mathcal T_m^{\mathtt A}(n):=\partial\mathcal B_m^{\mathtt A}(n)\setminus\mathcal E_m^{\mathtt A}(n),\ m\ge5,\ n\in\N,
  \end{gather}
  and they compute as 
  \begin{align}
   \mathcal T_m^{\mathtt b}&=\Btwn\big((-\frac{\plannedwidth_m}{16}+1,\frac{\plannedwidth_m}{16}),(\frac{3\plannedwidth_m}{16},\frac{\plannedwidth_m}{16})\big),\\
   \mathcal T_m^{\mathtt c}&=\Btwn\big((\frac{3\plannedwidth_m}{16},\frac{\plannedwidth_m}4-1),(\frac{11\plannedwidth_m}{16},\frac{\plannedwidth_m}4-1)\big),\\
   \mathcal T_m^{\mathtt A}(n)&=\Btwn\big((n+1,-\frac{\plannedwidth_m}{16}+1),(n+1,\frac{\plannedwidth_m}{16})\big),\\
   \mathcal T_m^{\mathtt B}&=\Btwn\big((\frac{\plannedwidth_m}{16},-\frac{\plannedwidth_m}{16}+1),(\frac{\plannedwidth_m}{16},\frac{3\plannedwidth_m}{16})\big),\\
   \mathcal T_m^{\mathtt C}&=\Btwn\big((\frac{\plannedwidth_{m+1}}4-1,\frac{3\plannedwidth_m}{16}),(\frac{\plannedwidth_{m+1}}4-1,\frac{\plannedwidth_{m+1}}2+\frac{3\plannedwidth_m}{16})\big),\ m\ge5,\ n\in\N.
  \end{align}
  Visualizations of these events can be found in Figures \ref{escapeevents1} and \ref{escapeevents2}. There, also the events of interest and the environments in the following Lemmata are shown.
 \end{dfn} 
 \begin{lem}\thlabel{escapen2}
  Define an environment by setting, for $u\in\Z^2$,
  \begin{equation}
  \varpi^{\mathtt b}(u):=\begin{cases}
                \omega_{\nwarrow}&\text{if }u_1<0,\ u_0>0,\\
                \omega_{\nearrow}&\text{else}
              \end{cases}
  \end{equation}
  which engenders the random walk $(X_t)$ starting in $v$ under $P_v^{\varpi^{\mathtt b}}$, $v\in\Z^2$. Let $v_m\in\mathcal S_m^{\mathtt b}$, $m\ge5$, be an arbitrary sequence. It then holds that $P_{v_m}^{\varpi^{\mathtt b}}(\tau_{\mathcal E_m^{\mathtt b}}<\tau_{\mathcal T_m^{\mathtt b}})$ is summable in $m$.
 \end{lem}
 \begin{proof}
  The arguments will be quite the same as in the proof of the last \thnameref{escapen1}.  
 
  There are four possibilities of escape to $\mathcal E_m^{\mathtt b}$, namely 
  \begin{itemize}
   \item to the south, which is exponentially becoming unlikely as the box grows with $m$, because of the uniform drift to the north.
   \item to the west, which is exponentially becoming unlikely because of the uniform drift pushing in the opposite direction on the western half--plane.
   \item to the east, which is exponentially becoming unlikely because the drift to the north is in the eastern half--plane at least as strong as the drift to the east, which means that the  linear speed of $X_{\cdot,1}$ is at least the same as the one of $X_{\cdot,0}$. With the box growing large, even if $X_\cdot$ starts at the easternmost possible point $\frac{\plannedwidth_m}{16}e_0$, by the time $X_{\cdot,1}$ reaches $\frac{\plannedwidth_m}{16}$, $X_{\cdot,0}$ will not have reached $\frac{3\plannedwidth_m}{16}+1$.
   \item to the horizontal piece of $\partial B_a$ in the northern west, which is exponentially becoming unlikely because the drift to the north provides that the probability of $X_{\cdot,1}$ being smaller than $0$ at the time $X_{\cdot,0}$ hits $\frac{\plannedwidth_m}{16}$ is decaying fastly.
  \end{itemize}
 \end{proof}

 \begin{lem}\thlabel{escapen3}
  $P_{v_m}^{\omega_{\nearrow}}(\tau_{\mathcal E_m^{\mathtt c}}<\tau_{\mathcal T_m^{\mathtt c}})$ is summable in $m$ for any arbitrary sequence $v_m\in\mathcal S_m^{\mathtt c}$, $m\ge5$.
 \end{lem}

 \begin{lem}\thlabel{escapen4}
  Take some sequence $(n_m)_{m\ge5}$ such that $\frac{\plannedwidth_m}2\le n_m\le\plannedwidth_{m+2}^\alpha$, $m\ge5$. Also take a sequence of starting points $v_m\in\mathcal S_m^{\mathtt A}$, $m\ge5$. Define the environment by setting
 \begin{equation}
  \varpi^{\mathtt A}(u):=\begin{cases}
               \omega_{\searrow}&\text{if }u_1>0,\\
               \omega_{\nearrow}&\text{if }u_1\le0,\ u\in\Z^2.
             \end{cases}
 \end{equation}
 It now holds that $P_{v_m}^{\varpi^{\mathtt A}}(\tau_{\mathcal E_m^{\mathtt A}(n_m)}<\tau_{\mathcal T_m^{\mathtt A}(n_m)})$ is summable in $m$.
 \end{lem}
 \begin{lem}\thlabel{escapen5}
  Define an environment by setting, for $u\in\Z^2$,
  \begin{equation}
  \varpi^{\mathtt B}(u):=\begin{cases}
                \omega_{\searrow}&\text{if }u_0<0,\ u_1>0,\\
                \omega_{\nearrow}&\text{else},\\
              \end{cases}
  \end{equation}
  which engenders the random walk $(X_t)$ starting in $v$ under $P_v^{\varpi^{\mathtt B}}$, $v\in\Z^2$. Let $v_m\in\mathcal S_m^{\mathtt B}$ be an arbitrary sequence. It then holds that $P_{v_m}^{\varpi^{\mathtt B}}(\tau_{\mathcal E_m^{\mathtt B}}<\tau_{\mathcal T_m^{\mathtt B}})$ is summable in $m$.
 \end{lem}
 \begin{lem}\thlabel{escapen6}
  $P_{v_m}^{\omega_{\nearrow}}(\tau_{\mathcal E_m^{\mathtt C}}<\tau_{\mathcal T_m^{\mathtt C}})$ is summable in $m$ for any arbitrary sequence $v_m\in\mathcal S_m^{\mathtt C}$.
 \end{lem}
 The arguments needed for the proofs of these last four \thnameref{escapen3}ta are the same as in the two preceeding proofs, which is why we omit them here. 
\begin{figure}[tbh]
\begin{center}
\begin{tikzpicture}[scale=0.6]
\def\escapedecor{[red,decorate,decoration={zigzag,segment length=1.5mm,amplitude=0.5mm}]}
\def\uebergang[#1]#2{\begin{tikzpicture}[rotate=#1,xscale=#2, scale=0.5]\draw[<->](-0.19,0)--(0.75,0);\draw[<->](0,-0.19)--(0,0.75); \filldraw[fill=black](0,0)circle(0.05cm);\end{tikzpicture}}
\begin{scope}
    \def\oneeigthm{1.2}
    \def\tothealpha{14}
    \draw[<->](-\oneeigthm,0)+(-0.5,0)--(\tothealpha+0.5,0);
    \draw[<->](0,-\oneeigthm-0.5)--(0,\oneeigthm+0.5);
    \draw (0,0) node[anchor=north east]{\strut$0$};
    \draw[ultra thick](0,0)--(\oneeigthm*8,0) +(-1,0)node[anchor=north east]{$\mathcal S_m^{\mathtt A}$};
    \draw[thick](-\oneeigthm,0)node[anchor=north east]{\strut$-\frac{\plannedwidth_m}{16}$}+(0,0.2) -- +(0,-0.2);
    \draw[thick](8*\oneeigthm,0)node[anchor=north west]{\strut$\frac{\plannedwidth_m}2$}+(0,0.2) -- +(0,-0.2);
    \draw[thick](\tothealpha,0) node[anchor=north east]{\strut$n$}+(0,0.2) -- +(0,-0.2);
    \draw[thick](0,-\oneeigthm-0.1) node[anchor=north west]{\strut$-\frac{\plannedwidth_m}{16}+1$}(0,-\oneeigthm)+(0.2,0) -- +(-0.2,0);
    \draw[thick](0,\oneeigthm+0.1) node[anchor=south west]{\strut$\frac{\plannedwidth_m}{16}$}(0,-\oneeigthm)+(0.2,0) -- +(-0.2,0);
    \draw[very thin](-\oneeigthm,-\oneeigthm) rectangle (\tothealpha,\oneeigthm);
    \draw\escapedecor(-\oneeigthm-0.1,-\oneeigthm)--(-\oneeigthm-0.1,\oneeigthm);
    \draw\escapedecor(-\oneeigthm,\oneeigthm+0.1)--(\tothealpha,\oneeigthm+0.1);
    \draw\escapedecor(-\oneeigthm,-\oneeigthm-0.1)--(\tothealpha,-\oneeigthm-0.1);
    \draw[green](\tothealpha+0.1,-\oneeigthm)--(\tothealpha+0.1,\oneeigthm);
    \draw(\tothealpha,1)node[anchor=west]{$\mathcal T_m^{\mathtt A}(n)$};
    \draw(5,-\oneeigthm-0.1)node[anchor=north]{$\mathcal E_m^{\mathtt A}(n)$};
    \draw(4,0.0)node[anchor=south west]{\uebergang[180]{-1}};
    \draw(4,-1.2)node[anchor=south west]{\uebergang[270]{-1}};
\end{scope}
\begin{scope}[shift={(1,-12.5)}]
      \def\onen{3}
      \fill[fill=green!20!white](0,-\onen) -- (0,\onen) -- (\onen,2*\onen) -- (\onen,0) -- cycle;
      \draw[<->](-\onen,0)+(-0.5,0)--(\onen+0.5,0);
      \draw[<->](0,-2*\onen-0.5)--(0,3*\onen+0.5);
      \draw (0,0) node[anchor=north east]{\strut$0$};
      \draw[thick](-\onen,0) node[anchor=north east]{\strut$-\frac{\plannedwidth_m}{16}$}+(0,0.2) -- +(0,-0.2);
      \draw[thick](\onen,0) node[anchor=north west]{\strut$\frac{\plannedwidth_m}{16}-1$}+(0,0.2) -- +(0,-0.2);
      \draw[thick](0,-2*\onen) node[anchor=south east]{\strut$-\frac{\plannedwidth_m}8+1$}+(0.2,0) -- +(-0.2,0);
      \draw[thick](0,-\onen) node[anchor=east]{\strut$-\frac{\plannedwidth_m}{16}+1$}+(0.2,0) -- +(-0.2,0);
      \draw[thick](0,\onen) node[anchor=east]{\strut$\frac{\plannedwidth_m}{16}$}+(0.2,0) -- +(-0.2,0);
      \draw[thick](0,3*\onen) node[anchor=north east]{\strut$\frac{3\plannedwidth_m}{16}$}+(0.2,0) -- +(-0.2,0);
      \draw[very thin](-\onen,-2*\onen) rectangle (\onen,3*\onen);
      \draw[ultra thick](0,-\onen)--(0,\onen)node[anchor=north east]{};
      \draw[green](\onen.1,-\onen)--(\onen.1,3*\onen);
      \draw\escapedecor(\onen.1,-2*\onen)--(\onen+0.1,-\onen);
      \draw\escapedecor(-\onen,-2*\onen-0.1)--(\onen,-2*\onen-0.1);
      \draw\escapedecor(-\onen.1,-2*\onen)--(-\onen.1,3*\onen);
      \draw\escapedecor(-\onen,3*\onen+0.1)--(\onen,3*\onen+0.1); 
      \draw(0,1.5)node[anchor=east]{$\mathcal S_m^{\mathtt B}$};
      \draw(\onen.1,1.5)node[anchor=west]{$\mathcal T_m^{\mathtt B}$};
      \draw(-\onen-0.1,1.5)node[anchor=east]{$\mathcal E_m^{\mathtt B}$};
      \draw(1,2)node[anchor=north west]{\uebergang[0]{1}};
      \draw(1,-2)node[anchor=south west]{\uebergang[0]{1}};
      \draw(-2.5,2)node[anchor=north west]{\uebergang[180]{-1}};
      \draw(-2.5,-2)node[anchor=south west]{\uebergang[270]{-1}};
\end{scope}
\begin{scope}[shift={(9,-17.5)}]
       \def\onehalf{12}
       \def\onefourth{6}
       \def\minusoneeighth{0.5}
       \def\threeminusoneeighth{1.5}
       \fill[fill=green!20!white](\minusoneeighth,\minusoneeighth*3)--(\minusoneeighth,\minusoneeighth*7) -- (\onefourth,\minusoneeighth*6+\onefourth) -- (\onefourth,\minusoneeighth*2+\onefourth) -- cycle;
       \draw[->](-0.2,0)--(\onefourth.5,0);
       \draw[->](0,-.2)--(0,2*\onefourth+3*\minusoneeighth+0.5);
       \draw (0,0) node[anchor=north east]{\strut$0$};
       \draw[very thin](0,0) rectangle (\onefourth,2*\onefourth+3*\minusoneeighth);
       \draw[ultra thick](\minusoneeighth,\minusoneeighth*3)--(\minusoneeighth,\minusoneeighth*7)node[anchor=north west]{$\mathcal S_m^{\mathtt C}$};
       \draw\escapedecor(0,-0.1)--(\onefourth,-0.1);
       \draw\escapedecor(-0.1,0)--(-0.1,2*\onefourth+3*\minusoneeighth);
       \draw\escapedecor(0,2*\onefourth+3*\minusoneeighth+0.1)--(\onefourth,2*\onefourth+3*\minusoneeighth+0.1);
       \draw\escapedecor(\onefourth+0.1,0)--(\onefourth+0.1,\minusoneeighth*3);
       \draw[green](\onefourth+0.1,\minusoneeighth*3)--(\onefourth+0.1,2*\onefourth+3*\minusoneeighth);
       \draw[thick](\minusoneeighth,-0.1) node[anchor=north]{\strut$\frac{\plannedwidth_m}{16}$}(\minusoneeighth,0)+(0,0.2) -- +(0,-0.2);
       \draw[thick](-0.1,\minusoneeighth*3) node[anchor=east]{$\frac{3\plannedwidth_m}{16}$}(0,\minusoneeighth*3)+(0.2,0) -- +(-0.2,0);
       \draw[thick](-0.1,\minusoneeighth*7) node[anchor=east]{$\frac{7\plannedwidth_m}{16}-1$}(0,\minusoneeighth*7)+(0.2,0) -- +(-0.2,0);
       \draw[thick](\onefourth,-0.1) node[anchor=north]{\strut$\frac{\plannedwidth_{m+1}}4-2$}(\onefourth,0)+(0,0.2) -- +(0,-0.2);
       \draw[thick](-0.1,2*\onefourth+3*\minusoneeighth) node[anchor=east]{$\frac{\plannedwidth_{m+1}}2+\frac{3\plannedwidth_m}{16}$}(0,2*\onefourth+3*\minusoneeighth)+(0.2,0) -- +(-0.2,0);
       \draw(\onefourth+0.1,3)node[anchor=west]{$\mathcal T_m^{\mathtt C}$};
       \draw(-0.1,\onefourth)node[anchor=east]{$\mathcal E_m^{\mathtt C}$};
       \draw(3,3)node[]{\uebergang[0]{1}};
\end{scope}
\end{tikzpicture}
\caption{Escape and target sets used in \thref{escapen4,escapen5,escapen6}.}\label{escapeevents2}
\end{center}
\end{figure}
\subsection{Proof of the Theorem}
We prove \thref{Xtoinfty} by showing that the random walk has positive probability to hit a certain sequence of target sets leading to infinity in a prescribed order, while not hitting the succession of escape--sets we define at the same time. The sets will be based on the ones who have just been treated in the \thref{escapen1,escapen2,escapen3,escapen4,escapen5,escapen6}. 
\begin{dfn}\thlabel{origins}
 We will shift the sets defined in \thref{manysets} by the vectors
 \begin{align}
  \mathcal O_m^{\mathtt a}&:=\UR\Lane_{+,+}(B_{m-1}^1)+(\frac{\plannedwidth_m}4,-\frac{\plannedwidth_{m-1}}{16}+1),\\
  \mathcal O_m^{\mathtt b}&:=\UR\Lane_{+,+}(B_m^0)+e_1,\\
  \mathcal O_m^{\mathtt c}&:=\UR\Lane_{+,+}(B_m^0)+(-\frac{\plannedwidth_m}4+1,1),\\
  \mathcal O_m^{\mathtt A}&:=\UR\Lane_{+,+}(B_m^0)+(-\frac{\plannedwidth_m}{16}+1,\frac{\plannedwidth_m}4),\\
  \mathcal O_m^{\mathtt B}&:=\UR\Lane_{+,+}(B_m^1)+e_0,\\
  \mathcal O_m^{\mathtt C}&:=\UR\Lane_{+,+}(B_m^1)+(1,-\frac{\plannedwidth_m}4+1),m\ge5;
 \end{align}
 ``$\mathcal O$'' stands for the shifted ``Origin''. 

 Also define 
 \begin{align}
  n_m^{\mathtt a}&:=(\UR\Lane_{+,+}(B_m^0))_1-(\UR\Lane_{+,+}(B_{m-1}^1))_1+\frac{\plannedwidth_{m-1}}{16}-1,\\
  n_m^{\mathtt A}&:=(\UR\Lane_{+,+}(B_m^1))_0-(\UR\Lane_{+,+}(B_m^0))_0+\frac{\plannedwidth_m}{16}-1,\ m\ge5.
 \end{align}
\end{dfn}
The next lemma shows how each shifted target set coincides with the next shifted starting set. 
\begin{lem}\thlabel{targetequalsstart}
 \begin{align}
  \mathcal T_m^{\mathtt a}(n_m^{\mathtt a})+\mathcal O_m^{\mathtt a}&=\mathcal S_m^{\mathtt b}+\mathcal O_m^{\mathtt b},\\
  \mathcal T_m^{\mathtt b}+\mathcal O_m^{\mathtt b}&=\mathcal S_m^{\mathtt c}+\mathcal O_m^{\mathtt c},\\
  \mathcal T_m^{\mathtt c}+\mathcal O_m^{\mathtt c}&=\mathcal S_m^{\mathtt A}+\mathcal O_m^{\mathtt A},\\
  \mathcal T_m^{\mathtt A}(n_m^{\mathtt A})+\mathcal O_m^{\mathtt A}&=\mathcal S_m^{\mathtt B}+\mathcal O_m^{\mathtt B},\\
  \mathcal T_m^{\mathtt B}+\mathcal O_m^{\mathtt B}&=\mathcal S_m^{\mathtt C}+\mathcal O_m^{\mathtt C},\\
  \mathcal T_m^{\mathtt C}+\mathcal O_m^{\mathtt C}&=\mathcal S_{m+1}^{\mathtt a}+\mathcal O_{m+1}^{\mathtt a}.
 \end{align}
\end{lem}
\begin{proof}
 We prove the first line, the others being similar. 
  \begin{align}
  \MoveEqLeft\mathcal T_m^{\mathtt a}(n_m^{\mathtt a})+\mathcal O_m^{\mathtt a}\\
   &=\Btwn\big((-\frac{\plannedwidth_m}{16}+1,n_m^{\mathtt a}+1),(\frac{\plannedwidth_m}{16},n_m^{\mathtt a}+1)\big)+\UR\Lane_{+,+}(B_{m-1}^1)+(\frac{\plannedwidth_m}4,-\frac{\plannedwidth_{m-1}}{16}+1)\\
   &=\Btwn\big((-\frac{\plannedwidth_m}{16}+1,(\UR\Lane_{+,+}(B_m^0))_1-(\UR\Lane_{+,+}(B_{m-1}^1))_1+\frac{\plannedwidth_{m-1}}{16}),\\
   &\qquad\qquad\qquad\quad(\frac{\plannedwidth_m}{16},(\UR\Lane_{+,+}(B_m^0))_1-(\UR\Lane_{+,+}(B_{m-1}^1))_1+\frac{\plannedwidth_{m-1}}{16})\big)\\
   &\qquad\qquad\qquad\qquad\qquad\qquad\qquad\qquad\qquad+\UR\Lane_{+,+}(B_{m-1}^1)+(\frac{\plannedwidth_m}4,-\frac{\plannedwidth_{m-1}{16}}8+1)\\
   &=\Btwn\big(((\UR\Lane_{+,+}(B_{m-1}^1))_0+\frac{\plannedwidth_m}4-\frac{\plannedwidth_m}{16}+1,(\UR\Lane_{+,+}(B_m^0))_1+1),\\
   &\qquad\qquad\quad((\UR\Lane_{+,+}(B_{m-1}^1))_0+\frac{\plannedwidth_m}4+\frac{\plannedwidth_m}{16},(\UR\Lane_{+,+}(B_m^0))_1+1)\big)\\
   &=\Btwn\Big(\big((\UR\Lane_{+,+}(B_m^0))_0-\frac{\plannedwidth_m}{16}+1,(\UR\Lane_{+,+}(B_m^0))_1+1\big),\\
   &\qquad\qquad\qquad\big((\UR\Lane_{+,+}(B_m^0))_0+\frac{\plannedwidth_m}{16},(\UR\Lane_{+,+}(B_m^0))_1+1\big)\Big)\\
   &=\Btwn\big((-\frac{\plannedwidth_m}{16}+1)e_0,\frac{\plannedwidth_m}{16}e_0\big)+\UR\Lane_{+,+}(B_m^0)+e_1=\mathcal S_m^{\mathtt b}+\mathcal O_m^{\mathtt b}.
 \end{align}
\end{proof}
\begin{proof}[Proof of \thref{Xtoinfty}] 
 Out of convenience, we set 
 \begin{equation}
  \mathcal T_m^\dagger:=\mathcal T_m^\dagger(n_m^\dagger),\ \mathcal E_m^\dagger:=\mathcal E_m^\dagger(n_m^\dagger),\ \mathcal B_m^\dagger:=\mathcal B_m^\dagger(n_m^\dagger),\ \dagger\in\{\text{``$\mathtt a$'',``$\mathtt A$''}\},\ m\ge5,
 \end{equation}
 and $\ABC:=\{\text{``$\mathtt a$'',``$\mathtt b$'',``$\mathtt c$'',``$\mathtt A$'',``$\mathtt B$'',``$\mathtt C$''}\}$. Also define the initial target-- and escape--sets 
 \begin{equation}
  \mathcal T^0:=\mathcal S_{M+1}^{\mathtt a}\mathcal O_{M+1}^{\mathtt a}\text{ and }\mathcal E^0:=\big(\partial\Btwn(0,\UR\mathcal T^0-e_0)\big)\setminus\mathcal T^0.
 \end{equation}
 The event 
 \begin{equation}
  \{\tau_{\mathcal T^0}<\tau_{\mathcal E^0}\}\cap\bigcap_{m\ge M+1}\bigcap_{\dagger\in\ABC}\big\{\tau(\mathcal S_m^\dagger+\mathcal O_m^\dagger,\mathcal T_m^\dagger+\mathcal O_m^\dagger)<\tau(\mathcal S_m^\dagger+\mathcal O_m^\dagger,\mathcal E_m^\dagger+\mathcal O_m^\dagger)\big\}
 \end{equation}
 implies $X_t\cdot\vec1\to\infty$, $t\to\infty$: it describes the path of a random walk that hits a target set, from this target set moves to the next target set, and so on. As, roughly speaking, these target sets ``lead to infinity in the direction of the vector $\vec1=(1,1)$'', they help describing a path of a random walk the scalar product with $\vec 1$ of which is diverging to $+\infty$. A picture of a piece of such a path with the corresponding target sets is available in Figure \ref{pathandtargets}.
 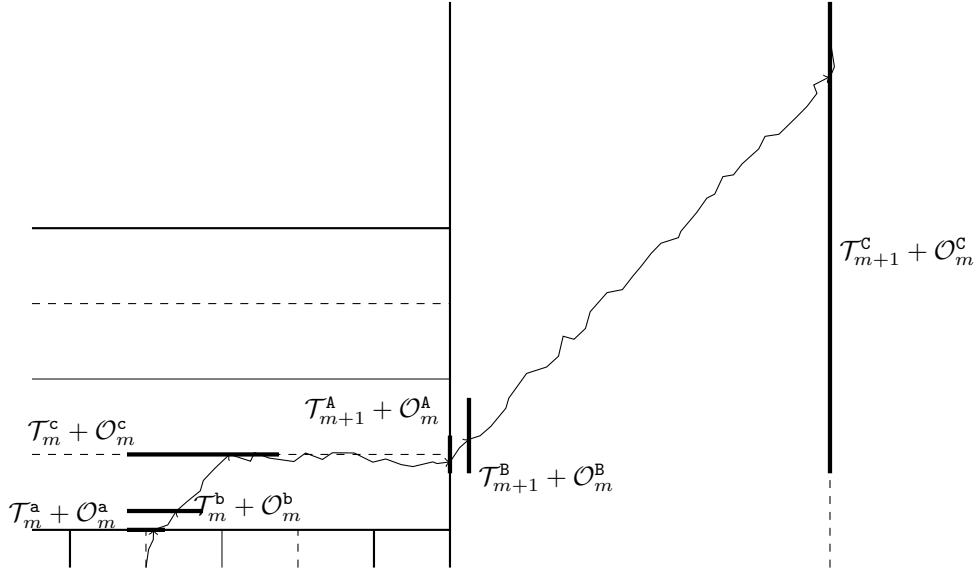
\begin{figure}[htb]
  \begin{center}
   \begin{tikzpicture}
    \foreach \shifter in {-1cm,3cm}\draw[xshift=\shifter,thick](0,-0.5)--(0,0);
    \foreach \shifter in {0cm,2cm}\draw[xshift=\shifter,dashed](0,-0.5)--(0,0);
    \foreach \shifter in {1cm}\draw[xshift=\shifter, very thin](0,-0.5)--(0,0);
    \foreach \shifter in {0cm,4cm}\draw[yshift=\shifter,thick](-1.5,0)--(4,0);
    \foreach \shifter in {2cm}\draw[yshift=\shifter, very thin](-1.5,0)--(4,0);
    \foreach \shifter in {1cm,3cm}\draw[yshift=\shifter,dashed] (-1.5,0)--(4,0);
    \foreach \shifter in {0cm}\draw[xshift=\shifter,thick](4,-0.5)--(4,7);
    \foreach \shifter in {5cm}\draw[xshift=\shifter,dashed](4,-0.5)--(4,7);
    \begin{scope}[decoration={random steps,segment length=2mm}]
     \draw[decorate,->](0,-0.5)--(0.1,0);
     \draw[decorate,->](0.1,0)--(0.4,0.25);
     \draw[decorate,->](0.4,0.25)--(1.1,1);
     \draw[decorate,->](1.1,1)--(4,0.9);
     \draw[decorate,->](4,0.9)--(4.25,1.2);
     \draw[decorate,->](4.25,1.2)--(9,6);
     \draw[decorate](9,6)--(9,7);
    \end{scope}
    \draw[ultra thick](-0.25,0)--(0.25,0)(-0.25,0.25)--(0.75,0.25)(-0.25,1)--(1.75,1)(4,0.75)--(4,1.25)(4.25,0.75)--(4.25,1.75)(9,0.75)--(9,7);
    \draw(-0.25,-0.1)node[anchor=south east]{$\mathcal T_m^{\mathtt a}+\mathcal O_m^{\mathtt a}$};
    \draw(0.5,0)node[anchor=south west]{$\mathcal T_m^{\mathtt b}+\mathcal O_m^{\mathtt b}$};
    \draw(0,1)node[anchor=south east]{$\mathcal T_m^{\mathtt c}+\mathcal O_m^{\mathtt c}$};
    \draw(4,1.25)node[anchor=south east]{$\mathcal T_{m+1}^{\mathtt A}+\mathcal O_m^{\mathtt A}$};
    \draw(4.25,1)node[anchor=north west]{$\mathcal T_{m+1}^{\mathtt B}+\mathcal O_m^{\mathtt B}$};
    \draw(9,4)node[anchor=north west]{$\mathcal T_{m+1}^{\mathtt C}+\mathcal O_m^{\mathtt C}$};
   \end{tikzpicture}
  \end{center}
  \caption{Target areas. The path has positive probability to hit them in that order.}\label{pathandtargets}
 \end{figure}

 With the help of \thref{targetequalsstart}, we can successively apply the strong Markov property for $X_\cdot$, and see that $P$--a.s.,
 \begin{align}
  \MoveEqLeft P_0^\omega(X_t\cdot\vec1\xrightarrow[t\to\infty]{}\infty)\\
  &\ge P_0^\omega\Big(\{\tau_{\mathcal T^0}<\tau_{\mathcal E^0}\}\cap\bigcap_{m\ge M+1}\bigcap_{\dagger\in\ABC}\big\{\tau(\mathcal S_m^\dagger+\mathcal O_m^\dagger,\mathcal T_m^\dagger+\mathcal O_m^\dagger)<\tau(\mathcal S_m^\dagger+\mathcal O_m^\dagger,\mathcal E_m^\dagger+\mathcal O_m^\dagger)\big\}\Big)\\
  &=P_0^\omega\big(\tau_{\mathcal T^0}<\tau_{\mathcal E^0}\big)\prod_{m\ge M+1}\smashoperator[r]{\prod_{\dagger\in\ABC}}P_0^\omega\Big(\tau\big(\mathcal S_m^\dagger+\mathcal O_m^\dagger,\mathcal T_m^\dagger+\mathcal O_m^\dagger\big)<\tau\big(\mathcal S_m^\dagger+\mathcal O_m^\dagger,\mathcal E_m^\dagger+\mathcal O_m^\dagger\big)\Big).
 \end{align}
 Because of the ellipticity of the random environment, and because $M$ from \eqref{dfnM} is $P$--a.s.\ finite, the first probability on the right hand side is strictly larger than $0$.

 The product being larger than $0$ is thus equivalent to 
 \begin{align}
  \sum_{\dagger\in\ABC}\smashoperator[r]{\sum_{m\ge M+1}}P_0^\omega\Big(\tau\big(\mathcal S_m^\dagger+\mathcal O_m^\dagger,\mathcal T_m^\dagger+\mathcal O_m^\dagger\big)>\tau\big(\mathcal S_m^\dagger+\mathcal O_m^\dagger,\mathcal E_m^\dagger+\mathcal O_m^\dagger\big)\Big)<\infty.
 \end{align}
 The case ``$=$'' cannot occur because the target-- and escape--sets are disjoint.
 Hence, what we need to show is the $P$--almost sure summability in $m$ of 
 \begin{equation}
  P_0^\omega\Big(\tau\big(\mathcal S_m^\dagger+\mathcal O_m^\dagger,\mathcal E_m^\dagger+\mathcal O_m^\dagger\big)<\tau\big(\mathcal S_m^\dagger+\mathcal O_m^\dagger,\mathcal T_m^\dagger+\mathcal O_m^\dagger\big)\Big),\ \dagger\in\ABC.
 \end{equation}
 Let us look at the case $\dagger=\text{``$\mathtt b$''}$. Note that 
 \begin{equation}
  X_t\in\mathcal B_m^{\mathtt b}+\mathcal O_m^{\mathtt b}\text{ for all }t\in\Big\{\tau\big(\mathcal S_m^{\mathtt b}+\mathcal O_m^{\mathtt b}\big),\dots,\big[\tau\big(\mathcal S_m^{\mathtt b}+\mathcal O_m^{\mathtt b},\mathcal T_m^{\mathtt b}+\mathcal O_m^{\mathtt b}\big)\wedge\tau\big(\mathcal S_m^{\mathtt b}+\mathcal O_m^{\mathtt b},\mathcal E_m^{\mathtt b}+\mathcal O_m^{\mathtt b}\big)\big]-1\Big\}.
 \end{equation}
 Also, $\omega(u)=(\theta_{\mathcal O_m^{\mathtt b}}\varpi^{\mathtt b})(u)$ for all $u\in\mathcal B_m^{\mathtt b}+\mathcal O_m^{\mathtt b}$, where $\varpi^{\mathtt b}$ is the one defined in \thref{escapen2}. This is true because of the placements of $\mathcal O_m^{\mathtt b}$, and \thref{widthandlength}. 

 So, the probability is the same as the one in \thref{escapen2}, which yields summability. 

 The other cases in $\ABC$ can be treated the same way using \thref{escapen1,escapen3,escapen4,escapen5,escapen6}; for ``$\mathtt a$'' and ``$\mathtt A$'', we need to remark that $(n_m^{\mathtt a})_m$ and $(n_m^{\mathtt A})_m$ satisfy the necessary conditions.
\end{proof} 


\begin{thebibliography}{BZZ06}

\bibitem[BZZ06]{BramsonZeitouniZerner06}
Maury Bramson, Ofer Zeitouni, and Martin~P.W. Zerner.
\newblock {Shortest spanning trees and a counterexample for random walks in
  random environments.}
\newblock {\em Ann. Probab.}, 34(3):821--856, 2006.

\bibitem[Guo]{guo}
Xiaoqin Guo.
\newblock {On the limiting velocity of random walks in mixing random
  environment.}
\newblock {Preprint.}

\bibitem[HM09]{HaggstromMester09}
Olle H{\"a}ggstr{\"o}m and P{\'e}ter Mester.
\newblock Some two-dimensional finite energy percolation processes.
\newblock {\em Electron. Commun. Probab.}, 14:42--54, 2009.

\bibitem[HS]{HolmesSalisbury11}
Mark Holmes and Thomas~S. Salisbury.
\newblock {Degenerate Random Walks in Random Environment.}
\newblock {Preprint.}

\bibitem[Kal81]{Kalikow81}
Steven~A. Kalikow.
\newblock Generalized random walk in a random environment.
\newblock {\em Ann. Probab.}, 9(5):753--768, 1981.

\bibitem[SZ99]{SznitmanZerner99}
Alain-Sol Sznitman and Martin Zerner.
\newblock A law of large numbers for random walks in random environment.
\newblock {\em Ann. Probab.}, 27(4):1851--1869, 1999.

\bibitem[Zer07]{Zerner07}
Martin~P.W. Zerner.
\newblock {The zero-one law for planar random walks in i.i.d. random
  environments revisited.}
\newblock {\em Electron. Commun. Prob.}, 12:326--335, 2007.

\bibitem[ZM01]{ZernerMerkl01}
Martin~P.W. Zerner and Franz Merkl.
\newblock {A zero-one law for planar random walks in random environment.}
\newblock {\em Ann. Probab.}, 29(4):1716--1732, 2001.

\end{thebibliography}
\end{document}